\documentclass[a4paper,11pt]{article}
\usepackage{amsmath, amsfonts, amscd, amssymb, amsthm, enumerate, xypic}

\def\bfB{\mathbf{B}}
\def\bfC{\mathbf{C}}
\def\ad{\text{ad}}

\newcommand{\botoplus}{\overset{\bot}{\oplus}}
\newcommand{\WS}{\mathcal{W}\mathcal{S}}
\newcommand{\WA}{\mathcal{W}\mathcal{A}}

\renewcommand{\epsilon}{\varepsilon}

\newcommand{\Mat}{\operatorname{M}}
\newcommand{\Mats}{\operatorname{S}}
\newcommand{\Mata}{\operatorname{A}}

\newcommand{\GL}{\operatorname{GL}}
\newcommand{\Ker}{\operatorname{Ker}}
\newcommand{\SKer}{\operatorname{SKer}}
\newcommand{\Diag}{\operatorname{Diag}}
\newcommand{\modu}{\operatorname{mod}}
\newcommand{\End}{\operatorname{End}}
\newcommand{\NT}{\operatorname{NT}}
\newcommand{\Vect}{\operatorname{span}}
\newcommand{\im}{\operatorname{Im}}
\newcommand{\Rad}{\operatorname{Rad}}
\newcommand{\tr}{\operatorname{tr}}
\newcommand{\rk}{\operatorname{rk}}
\newcommand{\codim}{\operatorname{codim}}
\renewcommand{\setminus}{\smallsetminus}

\def\F{\mathbb{F}}

\def\calA{\mathcal{A}}

\def\calM{\mathcal{M}}

\def\calS{\mathcal{S}}

\def\calU{\mathcal{U}}
\def\calV{\mathcal{V}}
\def\calW{\mathcal{W}}

\def\calZ{\mathcal{Z}}

\def\lcro{\mathopen{[\![}}
\def\rcro{\mathclose{]\!]}}

\theoremstyle{definition}

\theoremstyle{plain}
\newtheorem{theo}{Theorem}[section]
\newtheorem{prop}[theo]{Proposition}
\newtheorem{cor}[theo]{Corollary}
\newtheorem{lemma}[theo]{Lemma}
\newtheorem{claim}{Claim}

\theoremstyle{plain}

\theoremstyle{remark}
\newtheorem{Rems}{Remarks}[section]
\newtheorem{Rem}[Rems]{Remark}
\newtheorem{ex}[Rems]{Example}

\title{The structured Gerstenhaber problem (III)}
\author{Cl\'ement de Seguins Pazzis\footnote{Universit\'e de Versailles Saint-Quentin-en-Yvelines, Laboratoire de Math\'ematiques
de Versailles, 45 avenue des Etats-Unis, 78035 Versailles cedex, France, dsp.prof@gmail.com}}

\begin{document}


\thispagestyle{plain}

\maketitle

\begin{abstract}
Let $b$ be a symmetric bilinear form on a finite-dimensional vector space over a field with characteristic $2$.
Here, we determine the greatest possible dimension of a linear subspace of nilpotent $b$-symmetric or $b$-alternating endomorphisms of $V$, expressing it as a function of the dimension, the rank, the Witt index of $b$, and an additional invariant in a very special case.
\end{abstract}


\section{Introduction}

\subsection{The problem}

Throughout the article, $\F$ denotes an arbitrary field, $\overline{\F}$ denotes an algebraic closure of it, and we consider a
finite-dimensional vector space $V$ over $\F$, equipped with a bilinear form $b$
which is either symmetric, i.e.\ $\forall (x,y)\in V^2, \; b(x,y)=b(y,x)$,
or alternating, i.e.\ $\forall x \in V, \; b(x,x)=0$.
If $b$ is alternating, it is also skew-symmetric, i.e.\ $\forall (x,y)\in V^2, \; b(y,x)=-b(x,y)$, and in particular it is symmetric if
$\F$ has characteristic $2$. Conversely, every skew-symmetric bilinear form on $V$ is alternating if the characteristic of $\F$ is not $2$.
However, if $\F$ has characteristic
$2$ then on $V$ the skew-symmetric bilinear forms are the symmetric bilinear forms, and in general they differ from the alternating bilinear forms.

Given a subset $X$ of $V$, we denote by $X^{\bot_b}$ (or by $X^\bot$ if there is no possible confusion on the symmetric bilinear form $b$ under consideration) the set of all vectors $y \in V$ such that $x \underset{b}{\bot} y$, that is $b(x,y)=0$: it is a linear subspace of $V$. The radical of $b$ is defined as $\Rad(b):=V^\bot$. We say that $b$ is non-degenerate if $\Rad(b)=\{0\}$; in that case, $\dim X+\dim X^\bot=\dim V$
for every linear subspace $X$ of $V$. A linear subspace $X$ of $V$ is called \textbf{totally singular} when $X \subset X^\bot$, and the greatest
dimension of such a subspace is called the \textbf{Witt index} of $b$ and denoted by $\nu(b)$. A vector $x$ of $V$ is called $b$-isotropic whenever $b(x,x)=0$.
Finally, we say that $b$ is nonisotropic whenever its Witt index equals zero, i.e.\ $b(x,x)\neq 0$ for all $x \in V \setminus \{0\}$.

An endomorphism $u$ of $V$ is called $b$-symmetric (respectively, $b$-alternating)
whenever the bilinear form $(x,y) \in V^2 \mapsto b(x,u(y))$ is symmetric (respectively, alternating).
The set of all $b$-symmetric endomorphisms is denoted by $\calS_b$, and the set of all $b$-alternating
ones is denoted by $\calA_b$. Both sets are linear subspaces of the
space $\End(V)$ of all endomorphisms of $V$. Moreover, $\calA_b \subset \calS_b$ if $\F$ has characteristic $2$.

Finally, a subset of an $\F$-algebra is called \textbf{nilpotent} when all its elements are nilpotent.

\vskip 5mm
The standard Gerstenhaber problem consists of the following questions:
\begin{itemize}
\item What is the greatest dimension of a nilpotent (linear) subspace of $\End(V)$?
\item What are the nilpotent subspaces of $\End(V)$ with the greatest dimension?
\end{itemize}
These questions were raised by Gerstenhaber in \cite{Gerstenhaber}, and he answered them under a mild provision on the cardinality of the underlying field
(this provision was lifted later by Serezhkin \cite{Serezhkin}).
Below, we recall the answer to the first part of the problem:

\begin{theo}[Gerstenhaber (1958), Serezhkin (1985)]
Let $V$ be an $n$-dimensional vector space.
The greatest dimension of a nilpotent linear subspace of $V$ is $\frac{n(n-1)}{2}\cdot$
\end{theo}

See also \cite{Mathes} for an alternative proof, \cite{dSPGerstenhaberskew} for an extension to division rings,
and \cite{Quinlan,dSPlargerank} and \cite{DraismaKraftKuttler,MeshulamRadwan} for various generalizations.

\vskip 3mm
The present article deals with the equivalent of the standard Gerstenhaber problem for spaces equipped with bilinear forms.
Let $V$ be a finite-dimensional vector space equipped with a symmetric or alternating bilinear form $b$.
In the \emph{structured} Gerstenhaber problem, we ask the following questions:

\begin{itemize}
\item \textbf{Dimension question}: What is the greatest dimension of a nilpotent subspace of $\calS_b$ (respectively, of $\calA_b$?).
\item \textbf{Optimal spaces question:} What are the nilpotent subspaces of $\calS_b$ (respectively, of $\calA_b$) with the greatest dimension?
\end{itemize}

The study of the structured Gerstenhaber problem was initiated by Meshulam and Radwan \cite{MeshulamRadwan}, who tackled
the very special case of a non-degenerate symmetric bilinear form over an algebraically closed fields of characteristic $0$: in that situation they
answered the first question. The second question was given an answer by Draisma, Kraft and Kuttler \cite{DraismaKraftKuttler} for algebraically closed fields of characteristic different from $2$, as a special case of a general result on nilpotent linear subspaces of Lie algebras.
Recently, Kokol Bukov\v{s}ek and Omladi\v{c} \cite{BukovsekOmladic} rediscovered Meshulam and Radwan's result on spaces of symmetric nilpotent complex matrices, and they furthered the existing knowledge by answering the optimal spaces question in that situation (that is, $b$ is a non-degenerate symmetric bilinear form on a complex vector space, and one considers nilpotent subspaces of $b$-symmetric endomorphisms).

It is only very recently that a systematic treatment of the structured Gerstenhaber problem has been undertaken.
In \cite{dSPStructured1}, the dimension question was entirely answered over an arbitrary field of characteristic different from $2$,
and the optimal spaces question was also given a clear answer in two cases (when $b$ is symmetric and one considers spaces of $b$-symmetric endomorphisms, and when $b$ is alternating and one considers spaces of $b$-alternating endomorphisms).
In \cite{dSPStructured2}, the remaining cases for the optimal spaces question (that is, when $b$ is symmetric and one considers spaces of $b$-alternating endomorphisms, and when $b$ is alternating and one considers spaces of $b$-symmetric endomorphisms)
were successfully answered under mild cardinality assumptions on the underlying field.
Hence, in the characteristic different from $2$ case, the structured Gerstenhaber problem can be considered as essentially solved, with the exception of
some remaining issues over finite fields for the optimal spaces question.

However, there remains the problem of fields with characteristic $2$: indeed the techniques that were used in \cite{dSPStructured1}
fail for those fields. The main aim of the present article is to answer the dimension question over any field with characteristic $2$. In a subsequent article, we will tackle the optimal spaces question.

Before we go on, it is important to stress that we will only deal with non-degenerate bilinear forms. Indeed, it is known that, in order
to solve the structured Gerstenhaber theorem for a bilinear form $b$ on a vector space $V$, it suffices to have a solution
of that problem for the non-degenerate bilinear form $\overline{b}$ on $V/\Rad(b)$ induced by $b$ (then, one may combine this solution with the one of
the standard Gerstenhaber problem on the space $\Rad(b)$). For details, see section 1.1 of \cite{dSPStructured1}.

\subsection{The main result}\label{Section1.2}

It is time to recall the known answer to the dimension question over a field with characteristic different from $2$:

\begin{theo}[de Seguins Pazzis, theorem 1.7 of \cite{dSPStructured1}]\label{theocarnot2}
Let $V$ be a vector space with finite dimension $n$ over a field with characteristic different from $2$. Let $b$ be a non-degenerate
symmetric or alternating bilinear form on $V$, whose Witt index we denote by $\nu$.
\begin{enumerate}[(a)]
\item The greatest dimension of a nilpotent subspace of $\calS_b$ is~$\nu(n-\nu)$.
\item The greatest dimension of a nilpotent subspace of $\calA_b$ is~$\nu(n-\nu-1)$.
\end{enumerate}
\end{theo}

In order to state the corresponding result over fields with characteristic $2$, we need a few additional definitions.
From now on, we assume that $\F$ has characteristic $2$. Let $b$ be a symmetric bilinear form on $V$.
We denote by
$$Q_b : x \in V \mapsto b(x,x),$$
the corresponding quadratic form, and we simply denote it by $Q$ when there is no possible confusion.
Since $\F$ has characteristic $2$, we find that $Q$ is a group homomorphism from $(V,+)$ to $(\F,+)$,
in particular it has a kernel
$$\Ker Q :=\{x \in V : \; Q(x)=0\},$$
which is not only a subgroup of $(V,+)$ but also a linear subspace (because $Q(\lambda x)=\lambda^2 Q(x)$ for all $\lambda \in \F$ and $x \in V$).
The elements of $\Ker Q$ are called the \textbf{$b$-isotropic} vectors.
Then, we define
$$\SKer Q:=\Ker Q \cap (\Ker Q)^{\bot}=\{x \in V : \; Q(x)=0 \; \text{and} \; \forall y \in V, \; Q(y)=0 \Rightarrow b(x,y)=0\bigr\},$$
which is a linear subspace of $V$ (the notation $\SKer$ stands for ``super-kernel").

Now, we are ready to state our answer to the dimension question:

\begin{theo}[Main theorem]\label{maintheo}
Let $\F$ be a field of characteristic $2$.
Let $V$ be a finite-dimensional vector space over $\F$ and $b$ be a non-degenerate symmetric bilinear form
on $V$, with corresponding quadratic form denoted by $Q$. Denote by $\nu$ the Witt index of $b$, and set $n:=\dim V$.
\begin{enumerate}[(a)]
\item The greatest dimension of a nilpotent subspace of $\calS_b$ is $\nu (n-\nu)$.
\item If $n \neq 2\nu+1$ then the greatest dimension of a nilpotent subspace of $\calA_b$
is $\nu (n-\nu-1)$.
\item If $n=2\nu+1$ then the greatest dimension of a nilpotent subspace of $\calA_b$ is
$\nu (n-\nu)-\dim \SKer Q$.
\end{enumerate}
\end{theo}

Note that $\SKer Q$ is totally $b$-singular by its definition, and hence
$\dim \SKer Q \leq \nu$. It follows that
$$\nu (n-\nu-1) \leq \nu (n-\nu)-\dim \SKer Q \leq \nu (n-\nu).$$
Thus, if $n=2\nu+1$ and $\SKer Q$ has dimension $\nu$
then the bounds of (b) and (c) coincide and (b) applies with no condition on $n$ and $\nu$.
In general however, the case $n=2\nu+1$ of point (c) is special. At the other extreme, if $\SKer Q=\{0\}$, so that $b$
restricts to a non-degenerate form on $\Ker Q$, then the bound of (c) coincides with that of (a), and in this case the maximum
dimension of a nilpotent subspace of $\calS_b$ is attained by a subspace of $\calA_b$.

\begin{ex}
Assume that $\F$ has characteristic $2$.
Consider the standard scalar product
$$\bullet : (X,Y)\in (\F^n)^2 \mapsto X^T Y,$$
whose corresponding quadratic form is given by
$$Q : X \in \F^n \mapsto X^T X=(X\bullet E)^2$$
where $E \in \F^n$ is the vector with all entries equal to $1$.
Given $A \in \Mat_n(\F)$, we note that the endomorphism $X \mapsto AX$ of $\F^n$
is $\bullet$-symmetric (respectively, $\bullet$-alternating) if and only if $A$ is symmetric (respectively, symmetric with all diagonal entries zero).

Since $\bullet$ is non-degenerate, its Witt index $\nu$ satisfies $2\nu \leq n$.
Writing $p$ for $\lfloor n/2\rfloor$ and $f_i$ for $e_i+e_{i+p}$,  it is easily confirmed that the space
$\Vect(f_1,\dots,f_p)$ is totally $\bullet$-singular. Hence $\nu=p=\lfloor n/2\rfloor$.

In particular, if $n$ is even then $\nu=\frac{n}{2}$ and points (a) and (b) of Theorem \ref{maintheo} yield the following results:
\begin{itemize}
\item The greatest dimension of a nilpotent subspace of $n$ by $n$ symmetric matrices is $\frac{n^2}{4}\cdot$
\item The greatest dimension of a nilpotent subspace of $n$ by $n$ symmetric matrices with all diagonal entries zero is $\frac{n(n-2)}{4}\cdot$
\end{itemize}

Assume now that $n$ is odd. Then, $n=2\nu+1$.
Again, point (a) applies and it shows that the
greatest dimension of a nilpotent subspace of $n$ by $n$ symmetric matrices is $\frac{n^2-1}{4}$.
Here, we have $\Ker Q=\{E\}^\bot$ and $(\Ker Q)^\bot=\Vect(E)$, and since
$Q(E) \neq 0$ this yields $\SKer Q=\{0\}$. Hence, point (c) of Theorem \ref{maintheo} shows that $\frac{n^2-1}{4}$
is the greatest dimension of a nilpotent subspace of $n$ by $n$ symmetric matrices with all diagonal entries zero.
\end{ex}

The present article is entirely devoted to the proof of Theorem \ref{maintheo}, and it is organized as follows:
\begin{itemize}
\item In Section \ref{symformsreviewSection}, we quickly review the theory of symmetric bilinear forms over fields with characteristic $2$.
In particular, we study matrix representations, some results on the Witt index, and the notion of a normal basis.
\item In Section \ref{examplesSection}, we give examples of nilpotent subspaces that achieve the greatest possible dimension.
Such spaces are already known for the cases that correspond to points (a) and (b) in Theorem \ref{maintheo} (examples have already been
given in \cite{dSPStructured1}); here our main contribution lies in the special case when $n=2\nu+1$ and one considers $b$-alternating endomorphisms.
\item In Section \ref{toolsSection}, we collect various results on $b$-symmetric or $b$-alternating endomorphisms.
Two main objects of study are the $b$-symmetric and $b$-alternating endomorphisms of small rank, namely the $b$-symmetric squares and the
$b$-alternating tensors; and the so-called ``a-transform" of a symmetric bilinear form, a notion which appears to be new.
\item The remaining two sections are devoted to the proof of Theorem \ref{maintheo} \emph{per se}.
In Section \ref{m=0Section}, we tackle the special case when $\SKer Q=\Ker Q$: in that case the result is obtained as a relatively easy consequence of
Gerstenhaber's theorem. The remaining cases are dealt with in the last two sections:
 Section \ref{proofalternatingSection} for spaces of $b$-alternating endomorphisms, and
 Section \ref{proofsymmetricSection} for spaces of $b$-symmetric endomorphisms.
 In both cases, the proof is by induction on the dimension of the underlying vector space $V$. Key to this proof is the examination of elements of small rank in a nilpotent space of $b$-symmetric or $b$-alternating endomorphisms.
\end{itemize}

From now on, we always assume that $\F$ has characteristic $2$.

\section{A quick review of symmetric bilinear forms over fields with characteristic $2$}\label{symformsreviewSection}

\subsection{Matrix representations}

Given non-negative integers $n$ and $p$, we denote by $\Mat_{n,p}(\F)$ the vector space of all $n$ by $p$ matrices with entries
in $\F$, by $\Mat_n(\F)$ the algebra of all $n$ by $n$ matrices with entries in $\F$ (and by $I_n$ and $0_n$ its unity and zero element, respectively),
by $\NT_n(\F)$ its linear subspace consisting of all strictly upper-triangular matrices of $\Mat_n(\F)$ (whose dimension is $\dbinom{n}{2}$),
by $\Mats_n(\F)$ its linear subspace consisting of all symmetric matrices (whose dimension is $\dbinom{n+1}{2}$)
, and by $\Mata_n(\F)$ its linear subspace consisting of all alternating matrices (whose dimension is $\dbinom{n}{2}$).
We recall that a square matrix
$A \in \Mat_n(\F)$ is called \textbf{alternating} when the bilinear form $(X,Y) \mapsto X^T A Y$ on $(\F^n)^2$ is alternating, that is
$\forall X \in \F^n, \; X^TAX=0$. This means that $A$ is skew-symmetric with all diagonal entries equal to $0$.
Here, $\F$ has characteristic $2$ and hence alternating matrices are simply symmetric matrices with all diagonal entries $0$.

Let $b$ be a non-degenerate symmetric bilinear form whose attached quadratic form is denoted by $Q$
(possibly $Q=0$ if $b$ is alternating).

First of all, given a basis $(e_1,\dots,e_n)$ of $V$, we recall the definition of the matrix of $b$ with respect to $(e_1,\dots,e_n)$:
$$\Mat_{(e_1,\dots,e_n)}(b)=\bigl(b(e_i,e_j)\bigr)_{1 \leq i,j \leq n.}$$
This is a symmetric matrix, and it is alternating if and only if $b$ is an alternating form.
Note that $Q$ is easily computed from the matrix of $b$ with respect to $(e_1,\dots,e_n)$:
$$\forall (\lambda_1,\dots,\lambda_n)\in \F^n, \; Q\biggl(\sum_{k=1}^n \lambda_k\, e_k\biggr)=\sum_{k=1}^n \lambda_k^2\, b(e_k,e_k).$$
The additivity of $Q$ has already been noted in Section \ref{Section1.2}.

\subsection{On the Witt index}

In the proof of Theorem \ref{maintheo}, we will use several results on the Witt index of a symmetric bilinear form.
They are all based upon the following basic lemma (which works regardless of the characteristic of the field $\F$):

\begin{lemma}[Witt index additivity lemma]\label{Wittindexadd}
Let $b$ be a non-degenerate symmetric bilinear form on a finite-dimensional vector space $V$ over $\F$.
Consider a decomposition $V=V_1 \overset{\bot}{\oplus} V_2$ in which $\dim V_2=2$
and $V_2$ contains a non-zero isotropic vector. Denote by $b_1$ the restriction of $b$ to $V_1 \times V_1$.
Then, $\nu(b)=\nu(b_1)+1$.
\end{lemma}

\begin{proof}
We choose a non-zero isotropic vector $x$ in $V_2$.

Take a totally $b$-singular subspace $E$ of $V_1$ with dimension $\nu(b_1)$.
Then, $E \oplus \F x$ has dimension $\nu(b_1)+1$ and is totally $b$-singular, so $\nu(b) \geq \nu(b_1)+1$.

Conversely, suppose that there exists a totally $b$-singular subspace $U$ of $V$ whose dimension exceeds $\nu(b_1)+1$,
and set $U_1:=U \cap V_1$, which is totally $b$-singular. Then, $\dim(U_1) \leq \nu(b_1)$ and $\codim_U(U_1) \leq 2$.
We must have equality in both cases since $\dim U \geq \nu(b_1)+2$.

Since $\codim_U U_1=2$, the set $U \cap (V_1+z)$ is non-empty for each $z \in V_2$.
In particular, there exists $c \in V_1$ such that $c+x \in U$. Then, since $x$ is $b$-isotropic and $b$-orthogonal to $c$,
we obtain $b(c,c)=b(c+x,c+x)=0$. Moreover, for all $d \in U_1$ we find $b(c,d)=b(c+x,d)=0$.
Combining this with the fact that $U_1$ is totally $b$-singular, we gather that $U_1+\F c$ is a totally $b$-singular subspace of $V_1$.
Since $\dim U_1=\nu(b_1)$, it follows that $c \in U_1$ and hence $x \in U$. Then, for all $y \in V_2$,
we find some $d \in V_1$ such that $d+y \in U$, and we obtain $b(x,y)=b(x,d+y)=0$.
Hence $\{x\}^\bot$ includes $V_1$ and $V_2$, and we contradict the fact that $b$ is non-degenerate.
This contradiction completes the proof.
\end{proof}

By induction, we derive the following result:

\begin{cor}\label{Wittindexaddcor}
Let $b$ be a non-degenerate symmetric bilinear form on a finite-dimensional vector space $V$ over $\F$.
Consider a decomposition $V=W \overset{\bot}{\oplus} V_1 \overset{\bot}{\oplus} \cdots \overset{\bot}{\oplus} V_r$,
in which, for every $i \in \lcro 1,r\rcro$, $V_i$ has dimension $2$ and contains a non-zero isotropic vector.
Denote by $b_W$ the restriction of $b$ to $W^2$.
Then, $\nu(b)=\nu(b_W)+r$.
\end{cor}

\begin{cor}\label{quotientindexlemma}
Let $b$ be a non-degenerate symmetric bilinear form on a finite-dimensional vector space $V$ over $\F$.
Let $x \in V \setminus \{0\}$ be a non-zero $b$-isotropic vector. Consider the symmetric bilinear form $\overline{b}$
induced by $b$ on $\{x\}^\bot/\F x$. Then, $\nu(b)=\nu(\overline{b})+1$.
\end{cor}

\begin{proof}[Proof of Corollary \ref{quotientindexlemma}]
We choose $y \in V$ such that $b(x,y)=1$. Then $x,y$ are linearly independent. The bilinear form induced by $b$
on the $2$-dimensional space $P:=\Vect(x,y)$ is represented by the invertible matrix $\begin{bmatrix}
0 & 1 \\
1 & b(y,y)
\end{bmatrix}$, and hence it is non-degenerate. It follows that $V=P \botoplus P^\bot$.
Denote by $c$ the symmetric bilinear form induced by $b$ on $P^\bot$. The previous lemma shows that $\nu(b)=\nu(c)+1$.
Yet, $\F x \oplus P^\bot \subset \{x\}^\bot$, and since the dimensions are equal we obtain that $\F x \oplus P^\bot = \{x\}^\bot$.
Then, the canonical projection of $\{x\}^\bot$ on $\{x\}^\bot/\F x$ induces an isometry from $(P^\bot,c)$ to $(\{x\}^\bot/\F x,\overline{b})$,
which yields $\nu(\overline{b})=\nu(c)$. The claimed result ensues.
\end{proof}

\subsection{Normal bases}\label{normalbasisSection}

In this article, we shall make repeated use of the theory of symmetric bilinear forms over a field with characteristic $2$.
The following results are carefully explained in chapter XXXV of \cite{dSPinvitquad} and can also be
obtained with limited effort by using results from \cite{Milnor}. Before we go on, we warn the reader of a classical pitfall:
over fields with characteristic $2$, the theory of quadratic forms diverges from the one of symmetric bilinear forms.
In particular, the datum of $Q : x \mapsto b(x,x)$ is not sufficient to recover the symmetric bilinear form $b$, as
for every alternating bilinear form $c$ on $V$, the symmetric bilinear form $b+c$ also has $Q$ as its attached quadratic form!

Since $\F$ has characteristic $2$, every element $\alpha \in \overline{\F}$ has a unique square root in $\overline{\F}$,
and we denote this square root by $\sqrt{\alpha}$. Moreover, the mapping $\alpha \mapsto \sqrt{\alpha}$ is additive
since it is the inverse of the automorphism $\beta \mapsto \beta^2$ of the group $(\overline{\F},+)$.

Now, a \textbf{normal basis} for $b$ is a basis $(e_1,\dots,e_n)$ for which the matrix of $b$
has the form
$$\Diag(a_1,\dots,a_p) \oplus \begin{bmatrix}
0_q & I_q \\
I_q & \Diag(a_{p+1},\dots,a_{p+q})
\end{bmatrix} \oplus \begin{bmatrix}
0_m & I_m \\
I_m & 0_m
\end{bmatrix},$$
where $\sqrt{a_1},\dots,\sqrt{a_{p+q}}$ are linearly independent in the $\F$-vector space $\overline{\F}$,
and $p,q,m$ are non-negative integers such that $p+2q+2m=n$.
Asssume that we have such a basis. Then, many objects that are attached to $b$ can easily be computed.
First of all, we see that
$$\forall (\lambda_1,\dots,\lambda_n)\in \F^n, \; Q\Bigl(\sum_{k=1}^n \lambda_k.e_k\Bigr)=\sum_{k=1}^p \lambda_k^2\, a_k
+\sum_{k=1}^q \lambda_{p+q+k}^2\, a_{p+k}.$$
Since $\sqrt{a_1},\dots,\sqrt{a_{p+q}}$ are linearly independent in the $\F$-vector space $\overline{\F}$, it follows that
$$\Ker Q=\Vect(e_{p+1},\dots,e_{p+q},e_{p+2q+1},\dots,e_n).$$
Next, one deduces that $e_1,\dots,e_{p+q}$ all belong to $(\Ker Q)^\bot$, and since
$$\dim (\Ker Q)^\bot=n-\dim \Ker Q=n-(n-p-q)=p+q$$
and the vectors $e_1,\dots,e_{p+q}$ are linearly independent, it ensues that
$$(\Ker Q)^\bot=\Vect(e_1,\dots,e_{p+q}).$$
Hence
$$\SKer Q=(\Ker Q) \cap (\Ker Q)^\bot=\Vect(e_{p+1},\dots,e_{p+q}),$$
to the effect that
$$\dim (\SKer Q)=q.$$

Finally, Corollary \ref{Wittindexaddcor} applies to the decomposition
\begin{multline*}
V=\Vect(e_1,\dots,e_p) \botoplus \Vect(e_{p+1},e_{p+q+1}) \botoplus \cdots \botoplus \Vect(e_{p+q},e_{p+2q}) \\
 \botoplus \Vect(e_{p+2q+1},e_{p+2q+m+1}) \botoplus \cdots   \botoplus \Vect(e_{p+2q+m},e_{p+2q+2m})
\end{multline*}
and hence
$$\nu(b)=q+m+\nu(b_W)$$
where $W:=\Vect(e_1,\dots,e_p)$. Yet, since $\Ker Q=\Vect(e_{p+1},\dots,e_{p+q},e_{p+2q+1},\dots,e_n)$,
we see that $W$ contains no non-zero isotropic vector, which leads to $\nu(b_W)=0$.
We conclude that
$$\nu(b)=q+m.$$
In particular $n=2\nu(b)+1$ if and only if $p=1$.

We finish by recalling that $b$ always has a normal basis:

\begin{theo}[See \cite{dSPinvitquad} chapter XXXV Theorem 3.0.9]
Let $b$ be a non-degenerate symmetric bilinear form on a finite-dimensional vector space over a field $\F$
with characteristic $2$. Then, there exists a normal basis for $b$.
\end{theo}

\section{Examples of spaces with maximal dimension}\label{examplesSection}

\subsection{The matrix viewpoint}

Let $b$ be a symmetric bilinear form on a finite-dimensional vector space $V$.
Let $\bfB$ be a basis of $V$. Set $S:=\Mat_\bfB(b)$.

Given matrices $A \in \Mat_n(\F)$ and $M \in \Mat_n(\F)$, we say that
$M$ is \textbf{$A$-symmetric} (respectively, $A$-alternating) whenever $AM$ is symmetric (respectively, alternating).
We note by $\calS_A$ (respectively, by $\calA_A$) the vector space of all $A$-symmetric (respectively, $A$-alternating) matrices of $\Mat_n(\F)$.

Let $u \in \End(V)$, whose matrix in $\bfB$ is denoted by $\Mat_\bfB(u)$. Then, the
bilinear form $c : (x,y) \mapsto b(x,u(y))$ satisfies
$$\Mat_\bfB(c)=S \Mat_\bfB(u).$$
Hence, $u$ is $b$-symmetric (respectively, $b$-alternating) if and only if $\Mat_\bfB(u)$
is $S$-symmetric (respectively, $S$-alternating).

Hence,
$$u \in \calS_b \mapsto \Mat_\bfB(u) \in \calS_S \quad \text{and}
\quad u \in \calA_b \mapsto \Mat_\bfB(u) \in \calA_S$$
are nilpotency-preserving vector space isomorphisms.

\subsection{A general construction}\label{basicexamplesection}

Here, we assume that $b$ is non-degenerate.
Let us choose a normal basis $\bfB=(e_1,\dots,e_n)$ of $b$, so that
$$\Mat_\bfB(b)=\Diag(a_1,\dots,a_p) \oplus \begin{bmatrix}
0_q & I_q \\
I_q & \Diag(a_{p+1},\dots,a_{p+q})
\end{bmatrix} \oplus \begin{bmatrix}
0_m & I_m \\
I_m & 0_m
\end{bmatrix},$$
where $\sqrt{a_1},\dots,\sqrt{a_{p+q}}$ are linearly independent in the $\F$-vector space $\overline{\F}$,
and $p,q,m$ are non-negative integers such that $p+2q+2m=n$.
Denoting by $\nu$ the Witt index of $b$, we have
$\nu=q+m$ (see Section \ref{normalbasisSection}).

We consider the permuted basis
\begin{multline*}
\bfC:=(e_{p+1},\dots,e_{p+q},e_{p+2q+1},\dots,e_{p+2q+m},e_1,\dots,e_p,e_{p+q+1},\dots,e_{p+2q},\\
e_{p+2q+m+1},\dots,e_{p+2q+2m}).
\end{multline*}
With respect to $\bfC$, the matrix of $b$ equals
$$S:=\begin{bmatrix}
0_\nu & [0]_{\nu \times p} & I_\nu \\
[0]_{p \times \nu} & \Delta & [0]_{p \times \nu} \\
I_\nu & [0]_{\nu \times p} & \Theta
\end{bmatrix}$$
where $\Delta$ and $\Theta$ are diagonal matrices, with respective diagonal entries $a_1,\dots,a_p$ and $a_{p+1},\dots,a_{p+q},0,\dots,0$.
Since $\sqrt{a_1},\dots,\sqrt{a_p}$ are linearly independent over $\F$, the symmetric bilinear form
$(X,Y) \mapsto X^T\Delta Y$ is non-isotropic.

Let us consider a matrix $M \in \Mat_n(\F)$ of the following form:
$$M=\begin{bmatrix}
A & B & C \\
[0]_{p \times \nu} & D & E \\
0_\nu & F & G
\end{bmatrix},$$
i.e.\ $M$ represents, with respect to the basis $\bfC$, an endomorphism $u$ of $V$ that leaves the subspace~$W:=\Vect(e_{p+1},\dots,e_{p+q},e_{p+2q+1},\dots,e_{p+2q+m})$
invariant.
We compute that
$$SM=\begin{bmatrix}
0_\nu & F & G \\
[0]_{p \times \nu} & \Delta D & \Delta E \\
A & B+\Theta F & C+\Theta G
\end{bmatrix},$$
and hence $M$ is $S$-symmetric if and only if
$$F=0, \; G=A^T, \; B=(\Delta E)^T, \; D \; \text{is $\Delta$-symmetric} \quad \text{and} \quad
C+\Theta A^T \; \text{is symmetric}.$$
Likewise, $M$ is $S$-alternating if and only if
$$F=0, \; G=A^T, \; B=(\Delta E)^T, \; D \; \text{is $\Delta$-alternating} \quad \text{and} \quad
C+\Theta A^T \; \text{is alternating}.$$

Let us define $\WS_S$ (respectively $\WA_S$) as the vector space of all matrices of the form
$$\begin{bmatrix}
A & (\Delta E)^T & \Theta A^T+J \\
[0]_{p \times \nu} & 0_p & E \\
0_\nu & [0]_{\nu \times p} & A^T
\end{bmatrix}$$
where $A \in \NT_\nu(\F)$, $E \in \Mat_{p,\nu}(\F)$ and $J \in \Mats_\nu(\F)$
(respectively, $J \in \Mata_\nu(\F)$). Obviously,
$$\dim \WS_S=\dbinom{\nu}{2}+p\nu +\dbinom{\nu+1}{2}=\nu^2+p \nu=\nu(\nu+p)=\nu(n-\nu)$$
and
$$\dim \WA_S=\dim \WS_S-\nu=\nu(n-\nu-1).$$
Moreover, every matrix of $\WS_S$ (and in particular every matrix of $\WA_S$) is nilpotent
since its characteristic polynomial equals $t^n$. Finally, we have seen earlier that every matrix of
$\WS_S$ is $S$-symmetric, and every matrix of $\WA_S$ is $S$-alternating.

Therefore, as the vector space isomorphism $u \in \calS_b \mapsto \Mat_\bfC(u) \in \calS_S$
(respectively, $u \in \calA_b \mapsto \Mat_\bfC(u) \in \calA_S$) preserves nilpotency,
the inverse image of $\WS_S$ (respectively, of $\WA_S$) under it
is a nilpotent linear subspace of $\calS_b$ (respectively, of $\calA_b$) of dimension $\nu(n-\nu)$
(respectively, of dimension $\nu(n-\nu-1)$).

Hence, the claimed maximal dimension is reached in points (a) and (b) of Theorem \ref{maintheo}.
In the next two subsections, we tackle the upper-bound in point (c) of Theorem \ref{maintheo}.

\subsection{A special case}\label{specialexamplesection}

Here, we consider the special case when $n=2\nu(b)+1$ and $\SKer Q=\{0\}$.
The following example is a reformulation of one that was communicated to us by Rachel Quinlan.

Let us take a normal basis $\bfB=(e_1,\dots,e_n)$ for $b$, with corresponding indices $p,q,m$.
We know that $\nu(b)=q+m$ and $\dim \SKer Q=q$, and hence $q=0$ and $n=2m+1$. It follows that $p=1$.
Hence,
$$\Mat_\bfB(b)= \alpha I_1 \oplus \begin{bmatrix}
0_m & I_m \\
I_m & 0_m
\end{bmatrix}$$
for some non-zero scalar $\alpha \in \F$.
Better still, one checks that
$$\bfC:=(e_1,e_2,\dots,e_{m+1},\alpha e_{m+2},\dots,\alpha e_{2m+1})$$
is also a normal basis for $\alpha^{-1} b$, and
$$\Mat_\bfC(\alpha^{-1} b)= I_1 \oplus \begin{bmatrix}
0_m & I_m \\
I_m & 0_m
\end{bmatrix}.$$
Obviously $\calA_b=\calA_{\alpha^{-1} b}$, and hence it suffices to tackle the case when $\alpha=1$.

\begin{Rem}
It can be shown that in the reduced situation where $\alpha=1$, the form $b$ is represented by $I_n$ in some basis of
$V$, and hence it is isometric to the standard scalar product $(X,Y) \mapsto X^TY$ on $\F^n$.
\end{Rem}

In the next step, we show how one can construct various $b$-alternating nilpotent endomorphisms by starting from $A$-alternating
nilpotent matrices, where $A=\begin{bmatrix}
0_m & I_m \\
I_m & 0_m
\end{bmatrix}$. It is this key observation that lead us to discover the special bound in point (c) of Theorem \ref{maintheo}.

\begin{lemma}\label{extensionlemma}
Let $A \in \Mata_n(\F) \cap \GL_n(\F)$.
Set
$$S:=\begin{bmatrix}
1 & [0]_{1 \times n} \\
[0]_{n \times 1} & A
\end{bmatrix}.$$
Let $M \in \calS_A$ be nilpotent.
The following conditions are then equivalent:
\begin{enumerate}[(i)]
\item The matrix $M$ is $A$-alternating.
\item For all $X \in \F^n$, the matrix
$$M_X:=
\begin{bmatrix}
0 & X^TA^T \\
X & M
\end{bmatrix}$$
is nilpotent.
\end{enumerate}
Moreover, if condition (i) holds then $M_X$ is $S$-alternating for all $X \in \F^n$.
\end{lemma}

In particular, if $M$ is $A$-alternating and nilpotent, then $M_X$ is $S$-alternating and nilpotent for all $X \in \F^n$.

\begin{proof}
Let $X \in \F^n$.
We compute
$$SM_X=\begin{bmatrix}
0 & (AX)^T \\
AX & AM
\end{bmatrix}.$$
Hence, if $M$ is $A$-alternating, then $M_X$ is $S$-alternating for all $X \in \F^n$.

Next, we prove that conditions (i) and (ii) are equivalent.
Let $X \in \F^n$.
We note that the characteristic polynomial of $M_X$ satisfies
$$\chi_{M_X}(t)=\det(t I_{n+1}-M_X)=(\det S)^{-1}
\det(t S-SM_X)=(\det A)^{-1}
\det(t S-SM_X).$$
Noting that
$$t S-SM_X=\begin{bmatrix}
t & (AX)^T \\
AX & tA-AM
\end{bmatrix},$$
we obtain
$$\det(t S-SM_X)
=t \det(t A-AM)-(AX)^T (tA-AM)^{\ad} (AX).$$
where $N^\ad$ denotes the classical adjoint of the square matrix $N$.
Note that
$$(\det A)^{-1} (t\det(t A-AM))=t\,\det (t I_n-M).$$
Since $M$ is nilpotent, we have $\det(tI_n-M)=t^n$.
Therefore,
$$\chi_{M_X}(t)=t^{n+1}-(\det A)^{-1} (AX)^T (tA-AM)^{\ad} (AX).$$
It follows that $M_X$ is nilpotent if and only if $(AX)^T (tA-AM)^{\ad} (AX)=0$.
Since $A$ is invertible, we deduce that condition (ii) is equivalent to having $(tA-AM)^{\ad}$ alternating.

Yet $tA-AM$ is invertible as a matrix with entries in the field of fractions $\F(t)$, because its determinant equals $(\det A) \chi_M(t)$, a non-zero element of $\F(t)$.
It follows that $(tA-AM)^\ad$ is a nonzero scalar multiple of $(tA-AM)^{-1}$.
Noting that
$$\forall Y \in \F^n, \;    Y^T(tA-AM)^T Y=((tA-AM) Y)^T(tA-AM)^{-1} (tA-AM) Y,$$
we deduce that $(tA-AM)^{-1}$ is alternating if and only if $(tA-AM)^T$ is alternating, and hence if and only if
$tA-AM$ is alternating. Finally, since $tA$ is alternating,
$tA-AM$ is alternating if and only if $AM$ is alternating, i.e.\ $M$ is $A$-alternating.
Piecing the previous equivalences together, we conclude that condition (i) is equivalent to condition (ii).
\end{proof}

We are ready to construct our example. Set $A:=\begin{bmatrix}
0_m & I_m \\
I_m & 0_m
\end{bmatrix}$, which is invertible and alternating. Hence, its Witt index equals $m$.
As we have seen in Section \ref{basicexamplesection}, there exists a nilpotent linear subspace $\calV$ of $\calA_A$ with dimension
$m(m-1)$.
Let us consider the vector space $\widetilde{\calV}$ of all matrices of the form
$$\begin{bmatrix}
0 & (AX)^T \\
X & M
\end{bmatrix} \quad \text{where $M \in \calV$ and $X \in \F^{2m}$.}$$
By Lemma \ref{extensionlemma}, $\widetilde{\calV}$ is nilpotent and included in $\calA_{S}$ for $S:=I_1 \oplus A$.
Obviously,
$$\dim \widetilde{\calV}=2m+\dim \calV=m(m+1).$$
Hence, the inverse image of $\widetilde{\calV}$ under the isomorphism $u \in \End(V) \mapsto \Mat_\bfC(u) \in \Mat_n(\F)$
is a nilpotent linear subspace of $\calA_{\alpha^{-1} b}$ (that is, of $\calA_b$) with dimension
$m(m+1)$. Here, $m=\nu(b)$ and $n=2m+1$, and hence we have $\nu(b)(n-\nu(b))-\dim \SKer Q=m(m+1)$.
This yields a space whose dimension equals the one stated in point (c) of Theorem \ref{maintheo}.

\subsection{A general example for the case $n=2\nu(b)+1$}

Here, we combine the techniques of the previous two paragraphs to obtain
an example in the special case $n=2\nu(b)+1$. Again, we take a normal basis
$(e_1,\dots,e_n)$ for $b$, with associated indices $p,q,m$.
We assume that $n=2\nu(b)+1$, so that $p=1$.
We consider the permuted basis
$$\bfC:=(e_{2},\dots,e_{q+1},e_1,e_{2q+2},\dots,e_{2q+1+2m},e_{q+2},\dots,e_{2q+1}).$$
In that basis, the matrix of $b$ has the form
$$S=\begin{bmatrix}
0_q & [0]_{q \times (n-2q)} & I_q \\
[0]_{(n-2q) \times q} & \Delta & [0]_{(n-2q) \times q} \\
I_q & [0]_{q \times (n-2q)} & \Theta
\end{bmatrix}$$
where $\Theta \in \Mat_q(\F)$ is a diagonal matrix and
$$\Delta:=\alpha I_1 \oplus \begin{bmatrix}
0_m & I_m \\
I_m & 0_m
\end{bmatrix}$$
for some non-zero scalar $\alpha \in \F$.
As seen in the preceding section, there exists a nilpotent linear subspace $\calW$ of $\calA_\Delta$ with dimension $m(m+1)$.
We fix such a space. Now, we consider the space $\widetilde{\calW}$ of all matrices of the form
$$\begin{bmatrix}
A & (\Delta E)^T & \Theta A^T+J \\
[0]_{(n-2q) \times q} & N & E \\
0_q & [0]_{q \times (n-2q)} & A^T
\end{bmatrix}$$
where $A \in \NT_q(\F)$, $N \in \calW$, $E \in \Mat_{n-2q,q}(\F)$ and $J \in \Mata_q(\F)$.
Note that every such matrix is nilpotent as it is block-upper-triangular and all its diagonal blocks are nilpotent.
Just line in Section \ref{basicexamplesection}, one checks that every matrix in $\widetilde{\calW}$ is $S$-alternating.
Hence, $\widetilde{\calW}$ is a nilpotent linear subspace of $\calA_S$. Moreover,
$$\dim \widetilde{\calW}=\dbinom{q}{2}+q(n-2q)+\dbinom{q}{2}+\dim \calW=q(q-1)+q(n-2q)+m(m+1).$$
Since $q=\dim \SKer Q$, we obtain
\begin{align*}
\dim \widetilde{\calW}+\dim \SKer Q & =q^2+q(n-2q)+m(m+1) \\
& =q^2+q(2m+1)+m(m+1) \\
& =(q+m)(q+m+1)=\nu(n-\nu).
\end{align*}
Hence, the inverse image of $\widetilde{\calW}$ under $u \in \calA_b \mapsto \Mat_\bfB(u) \in \calA_S$
is a nilpotent linear subspace of $\calA_b$ with dimension $\nu(n-\nu)-\dim \SKer Q$.

\section{Preliminary results on $b$-symmetric endomorphisms}\label{toolsSection}

Throughout the section, we fix a non-degenerate symmetric bilinear form $b$
on a finite-dimensional vector space $V$ over $\F$. Remember that $\F$ is assumed to be of characteristic $2$.

\subsection{Three basic lemmas}

We start with two elementary observations: both appear also in section 1.3 of \cite{dSPStructured1}.

\begin{lemma}[Lemma 1.4 of \cite{dSPStructured1}]\label{lemmakerim}
Let $u$ be a $b$-symmetric endomorphism of $V$.
Then:
\begin{enumerate}[(a)]
\item For every linear subspace $W$ of $V$ that is stable under $u$, the subspace
$W^\bot$ is also stable under $u$.
\item We have $\Ker u=(\im u)^\bot$.
\end{enumerate}
\end{lemma}

\begin{proof}
Let $W \subset V$ be stable under $u$.
Let $y \in W^\bot$. For all $x \in W$ we have $b(x,u(y))=b(y,u(x))=0$ because $u(x) \in W$, and hence $u(y) \in W^\bot$.

For all $x \in V$, we have $x \in (\im u)^\bot$ if and only if $b(x,u(y))=0$ for all $y \in V$, i.e.\
$b(y,u(x))=0$ for all $y \in V$, and this is equivalent to $u(x)=0$ because $b$ is non-degenerate.
\end{proof}

\begin{lemma}[Lemma 1.5 of \cite{dSPStructured1}]\label{lemmanonisotropic}
Assume that $b$ is non-isotropic.
Let $u$ be a nilpotent $b$-symmetric endomorphism of $V$. Then $u=0$.
\end{lemma}

\begin{proof}
Assume that $u \neq 0$. Then $u^k=0$ for some minimal $k \geq 2$, and then $\im u^{k-1}$ is a non-zero subspace of both
$\im u$ and $\Ker u$. Using point (b) of Lemma \ref{lemmakerim}, we deduce that $\im u^{k-1}$ is totally $b$-singular, and we contradict
the fact that $b$ is non-isotropic.
\end{proof}

The next lemma is less easy: its proof is given in section 4.1 of \cite{dSPStructured2}.
Let us simply mention that the proof is by induction and involves Corollary \ref{quotientindexlemma}.

\begin{lemma}[See lemma 4.4 from \cite{dSPStructured2}]\label{stablesetimlemma}
Let $u$ be a nilpotent $b$-symmetric endomorphism of $V$.
Then, there exists a totally $b$-singular subspace of $V$ that has dimension $\nu(b)$ and is stable under $u$.
\end{lemma}

\subsection{On $b$-alternating and $b$-symmetric tensors}

First, denote by $V^\star$ the dual space of $V$. Given $x \in V$ and $\varphi \in V^\star$, we define the following endomorphism
of $V$:
$$\varphi \otimes x : \begin{cases}
V & \longrightarrow V \\
y & \longmapsto \varphi(y)\,x.
\end{cases}$$
We observe that:
\begin{itemize}
\item The trace of $\varphi \otimes x$ equals $\varphi(x)$;
\item Every endomorphism of rank $1$ of $V$ has the form $\varphi \otimes x$ for some non-zero elements $\varphi \in V^\star$
and $x \in V$;
\item For non-zero elements $\varphi,\varphi'$ of $V^\star$, and non-zero elements $x,x'$ of $V$,
the endomorphisms $\varphi \otimes x$ and $\varphi' \otimes x'$ are equal if and only if $\varphi'=\alpha \varphi$
and $x=\alpha x'$ for some $\alpha \in \F \setminus \{0\}$;
\item Since $b$ is non-degenerate, every endomorphism of $V$ of rank $1$ has the form $b(y,-) \otimes x$ for some pair
$(x,y)$ of non-zero vectors of $V$ (determined up to multiplication of $x$ by a non-zero scalar $\alpha$ and of $y$ by $\alpha^{-1}$).
\end{itemize}

Next, given $x \in V$, we note that $b(x,-) \in V^\star$ and that
$$x \otimes_b x:=b(x,-) \otimes x$$
is a $b$-symmetric endomorphism of $V$. We call it the \textbf{$b$-symmetric square of~$x$.}

\begin{lemma}\label{rank1lemma}
Let $u \in \calS_b$ have rank $1$. Then $u=\alpha\,x \otimes_b x$ for some non-zero $x \in V$ and some $\alpha \in \F \setminus \{0\}$.
\end{lemma}

\begin{proof}
Since $u$ has rank $1$, we have $\im u=\F x$ for some $x \in V \setminus \{0\}$.
Then $u=b(y,-) \otimes x$ for some $y \in V \setminus \{0\}$.
Since $u$ is $b$-symmetric, for all $s,t$ in $V$ we have
$$b\bigl(s,b(y,t)\,x\bigr)=b\bigl(t,b(y,s)\,x\bigr),$$
that is
$$b(y,t)\,b(s,x)=b(y,s)\,b(x,t).$$
If $y \not\in \F x$ then $\{x\}^\bot\not\subset \{y\}^\bot$ and hence by choosing $s \in \{x\}^\bot \setminus \{y\}^\bot$ and $t \in V \setminus \{x\}^\bot$
we obtain a contradiction from the above equality.
Hence $y=\alpha x$ for some $\alpha \in \F \setminus \{0\}$, and we conclude that $u=\alpha\,x \otimes_b x$.
\end{proof}

Given $x,y$ in $V$, one checks that the endomorphism
$$x \wedge_b y :=b(y,-)\otimes x-b(x,-)\otimes y$$
of $V$ is $b$-alternating: we call it the \textbf{$b$-alternating tensor of $x$ and $y$.}
We note that $x \wedge_b y$ is nonzero if and only if $x$ and $y$ are linearly independent, in which case
its range is $\Vect(x,y)$ since the non-degeneracy of $b$ guarantees that $b(y,-)$ and $b(x,-)$ are also linearly independent
in this situation.

Let $x$ be a non-zero element of $V$. Then the mapping
$$y \in V \mapsto x \wedge_b y \in \calA_b$$
is linear with kernel $\F x$.
Moreover, given $\alpha \in \F$ and $y \in V$, the equality $\alpha\, x \otimes_b x=x \wedge_b y$ holds if and only if $y \in \F x$
and $\alpha=0$. This is immediate from the observation that $x \otimes_b x$ has rank $1$ and $x \wedge_b y$ has rank $0$ or $2$.

The relevance of the above construction is emphasized by the following result:

\begin{prop}\label{caracsymmetrictensors}
Let $x$ be a non-zero $b$-isotropic vector of $V$.
Let $u \in \calS_b$. The following conditions are equivalent:
\begin{enumerate}[(i)]
\item The endomorphism $u$ vanishes at $x$ and maps $\{x\}^\bot$ into $\F x$.
\item There exist $y \in \{x\}^\bot$ and $\alpha \in \F$ such that $u=\alpha\,x \otimes_b x+x\wedge_b y$.
\end{enumerate}
Moreover, if they hold then $u$ is nilpotent and $u^3=0$.
\end{prop}

\begin{proof}
Assume first that $u=\alpha \,x \otimes_b x+x\wedge_b y$ for some $\alpha \in \F$
and some $y \in \{x\}^\bot$. Then, $u(x)=\alpha\, b(x,x)\,x+b(y,x)\,x-b(x,x)\,y=0$.
Moreover, for every $z \in \{x\}^\bot$ we have $u(z)=b(y,z)\,x$ because $b(x,z)=0$.
Hence, condition (i) is satisfied. Finally, noting that $\im u \subset \Vect(x,y) \subset \{x\}^\bot$ we observe that $\im u^2 \subset \F x$ and
hence that $u^3=0$.

Conversely, assume that (i) holds. Since $u$ maps $\{x\}^\bot$ into $\F x$, the rank theorem shows that
$\rk u \leq 2$ and if $\rk u=2$ then $\F x \subset \im u$. We consider in turn the three possibilities for $\rk u$.

If $u=0$ then (ii) holds with $\alpha=0$ and $y=0$.

Suppose $\rk u=1$. By Lemma \ref{rank1lemma}, $u=\alpha\,z \otimes_b z$ for some $z \in V \setminus \{0\}$
and some $\alpha \in \F \setminus \{0\}$. Then, $\im u=\F z$ and $\Ker u=\{z\}^\bot$.
Since $u$ has rank $1$ and $u(\{x\}^\bot) \subset \F x$, we either have
$x \in \im u=\F z$ or $\{x\}^\bot \subset \Ker u=\{z\}^\bot$. In either case $\F x = \F z$ and we conclude that
$u=\beta\,x \otimes_b x$ for some $\beta \in \F \setminus \{0\}$. Hence (ii) holds with $y=0$.

Assume finally that $\rk u=2$. Since $x \in \Ker u$, Lemma \ref{lemmakerim} yields that $\im u \subset \{x\}^\bot$.
Moreover, we have seen earlier that $x \in \im u$ because $\rk u=2$.
Hence, $\im u =\Vect(x,y')$ for some vector $y' \in \{x\}^\bot \setminus \F x$, and we can write
$u=\varphi \otimes x+\psi \otimes y'$ for some linear forms $\varphi,\psi$ on $V$.
The fact that $u$ maps $\{x\}^\bot$ into $\F x$ yields that $\psi$ vanishes everywhere on $\{x\}^\bot$, and hence
$\psi=\mu\, b(x,-)$ for some $\mu \in \F$. Finally, we have seen in the first part of the proof that
$\mu\,x \wedge_b y'$, which is $b$-symmetric, satisfies condition (i), and hence it is obvious that $v:=u+x \wedge_b (\mu y')$ also does.
Now, clearly $\im v \subset \F x$, and hence by the previous cases $v=\alpha\, x \otimes_b x$ for some $\alpha \in \F$,
whence $u=\alpha\, x \otimes_b x+x \wedge_b y$ for $y:=-\mu y' \in \{x\}^\bot$.
Hence, we have shown that condition (i) implies condition (ii).
\end{proof}

\begin{cor}\label{caracalternatingtensors}
Let $x$ be a non-zero $b$-isotropic vector of $V$.
Let $u \in \calA_b$. The following conditions are equivalent:
\begin{enumerate}[(i)]
\item The endomorphism $u$ vanishes at $x$ and maps $\{x\}^\bot$ into $\F x$.
\item There exists $y \in \{x\}^\bot$ such that $u=x\wedge_b y$.
\end{enumerate}
Moreover, if they hold then $u$ is nilpotent.
\end{cor}

\begin{proof}
We note that $x \wedge_b y$ is $b$-alternating for all $y \in V$.
In light of Proposition \ref{caracsymmetrictensors} and of the fact that $\calA_b$ is a linear subspace of $\calS_b$, it suffices to prove that the only scalar $\alpha$ for which $\alpha\,x \otimes_b x$ is $b$-alternating is $0$.
Let then $\alpha \in \F$. If $\alpha\,x \otimes_b x$ is $b$-alternating then
$$\forall z \in V, \; \alpha \,b(z,x)^2=0,$$
and hence $\alpha=0$ because we can choose $z \in V \setminus \{x\}^\bot$.
\end{proof}

\subsection{The a-transform of a symmetric bilinear form}

Denote by
$$Q : x \in V \mapsto b(x,x) \in \F$$
the quadratic form attached to $b$. We now view $Q$ as a mapping with codomain $\overline{\F}$.
Remember that $Q$ is additive. Since $Q(\lambda x)=\lambda^2 Q(x)$ for all $\lambda \in \F$ and all $x \in V$, we obtain that
$$(x,y)\in V^2 \mapsto \sqrt{Q(x)Q(y)}$$
is a symmetric $\F$-bilinear mapping from $V^2$ to $\overline{\F}$. Note that it is not a bilinear \emph{form} in general, as in general it does not map
into $\F$.
The symmetric $\F$-bilinear mapping
$$(x,y)\in V^2 \mapsto b(x,y)+\sqrt{Q(x)Q(y)} \in \overline{\F}$$
is called the \textbf{a-transform} of $b$, and we denote it by $b_a$.

One important feature of $b_a$ is that it is alternating: indeed, since $\F$ has characteristic $2$ we have
$$\forall x \in V, \; b_a(x,x)=b(x,x)+\sqrt{Q(x)^2}=2\, Q(x)=0.$$
Now we look at the matrix viewpoint.

Let $S=(s_{i,j}) \in \Mats_n(\F)$. We denote by
$$\Delta(S):=(s_{i,i})_{1 \leq i \leq n} \in \F^n$$
the \textbf{diagonal vector} of $S$, and by
$$\Delta(S)^{1/2}:=(\sqrt{s_{i,i}})_{1 \leq i \leq n} \in \overline{\F}^n$$
its square root.
The symmetric matrix
$$S_a:=S+\Delta(S)^{1/2} (\Delta(S)^{1/2})^T$$
belongs to $\Mats_n(\overline{\F})$: we call it the \textbf{a-transform} of $S$.
Better still, $S_a$ is alternating since for all $i \in \lcro 1,n\rcro$
its $i$-th diagonal entry equals $s_{i,i}+(\sqrt{s_{i,i}})^2=2 s_{i,i}=0$.

Now, let $\bfB:=(e_1,\dots,e_n)$ be a basis of $V$, and set
$S:=\Mat_\bfB(b)$.
Then, for all $(i,j)\in \lcro 1,n\rcro^2$,
$$b_a(e_i,e_j)=s_{i,j}+\sqrt{Q(e_i)Q(e_j)}=s_{i,j}+\Bigl(\Delta(S)^{1/2} (\Delta(S)^{1/2})^T\Bigr)_{i,j}$$
and hence
$$S_a=\bigl(b_a(e_i,e_j)\bigr)_{1 \leq i,j \leq n.}$$

Two vectors $x$ and $y$ of $V$ will be called $b_a$-orthogonal whenever $b_a(x,y)=0$.
More generally, given subsets $X$ and $Y$ of $V$, we say that $X$ is $b_a$-orthogonal to $Y$, and we write
$X \bot_{b_a} Y$, whenever $b_a(x,y)=0$ for all $x \in X$ and $y \in Y$.

\begin{lemma}[Orthogonality lemma for the a-transform]\label{atransformlemma}
Let $b$ be a non-degenerate symmetric bilinear form on an $n$-dimensional vector space $V$.
Let $V_1$ and $V_2$ be linear subspaces of $V$ such that $V_1 \bot_{b_a} V_2$.
Assume furthermore that $\dim V_1+\dim V_2>n$.
Then:
\begin{itemize}
\item $n$ is odd and the Witt index of $b$ equals $\frac{n-1}{2}\cdot$
\item Both $V_1$ and $V_2$ include $(\Ker Q)^\bot$.
\item $\dim V_1+\dim V_2=n+1$.
\end{itemize}
\end{lemma}

\begin{proof}
We start with an arbitrary basis $(e_1,\dots,e_n)$ of $V$.

For $i \in \{1,2\}$, we denote by $W_i$ the linear subspace of $\F^n$ consisting of the coordinate vectors of the elements of $V_i$ in $\bfB$,
and by $\overline{W_i}$ the $\overline{\F}$-linear subspace of $\overline{\F}^n$ spanned by the elements of $W_i$.
Classically, we have $\dim_{\overline{\F}} \overline{W_i}=\dim_{\F} W_i$.
We consider the alternating $\overline{\F}$-bilinear form
$$c : (X,Y) \in (\overline{\F}^n)^2 \mapsto X^T S_a Y.$$
The fact that $V_1$ is $b_a$-orthogonal to $V_2$ means that $W_1$ is $c$-orthogonal to $W_2$.
It follows that $\overline{W_1}$ is $c$-orthogonal to $\overline{W_2}$.
In the remainder of the proof, given a subset $X$ of $\overline{\F}^n$, we denote by $X^o$ its orthogonal complement with respect to
$c$.

Since $\dim_{\overline{\F}} \overline{W_1}+\dim_{\overline{\F}} \overline{W_2}=\dim_\F W_1+\dim_\F W_2>n$,
we find that the alternating form $c$ is degenerate, and hence the matrix $S_a$ is singular.

Yet, as $\Delta(S)^{1/2} (\Delta(S)^{1/2})^T$ has rank at most $1$, we have $\rk (S_a) \geq \rk S-1=n-1$.
Since $S_a$ is singular, we actually have $\rk(S_a)=n-1$. Moreover, since $S_a$ is alternating its rank is even, and hence
$n$ is odd.

Next, we claim that each one of $\overline{W_1}$ and $\overline{W_2}$ is the orthogonal complement of the other for $c$,
and in particular each one includes the kernel of $S_a$. Indeed, for every linear subspace $\calZ$ of $\overline{\F}^n$, we have the known formula
$$\dim_{\overline{\F}} \calZ^o+\dim_{\overline{\F}} \calZ=n+\dim_{\overline{\F}} (\calZ \cap \Rad(c)).$$
In particular, since $\Rad(c)=\Ker S_a$ has dimension $1$, we have
$\dim \overline{W_1}^o \leq n+1-\dim \overline{W_1} \leq \dim \overline{W_2}$; combining this with the converse inclusion
$\overline{W_2} \subset \overline{W_1}^o$ yields the equality $\overline{W_2}=\overline{W_1}^o$.
Symmetrically, $\overline{W_1}=\overline{W_2}^o$.

Now, we assume that $(e_1,\dots,e_n)$ is actually a normal basis for $b$, and we write $S$ as
$$S=\Diag(a_1,\dots,a_p) \oplus \begin{bmatrix}
0_q & I_q \\
I_q & \Diag(a_{p+1},\dots,a_{p+q})
\end{bmatrix} \oplus \begin{bmatrix}
0_m & I_m \\
I_m & 0_m
\end{bmatrix}$$
where the scalars $\sqrt{a_1},\dots,\sqrt{a_{p+q}}$ are linearly independent
in the $\F$-vector space $\overline{\F}$. Here, $n$ is odd, and hence so is $p$.

We define $X_0 \in \overline{\F}^n$ as
$$X_0:=\begin{bmatrix}
\sqrt{a_1}^{-1} & \cdots & \sqrt{a_p}^{-1} & \sqrt{a_{p+1}} & \cdots & \sqrt{a_{p+q}} & 0 & \cdots & 0
\end{bmatrix}^T.$$
We check that $X_0 \in \Ker S_a$ and $X_0 \neq 0$.
To see this, note first that
$$\bigl(\Delta(S)^{1/2}\bigr)^T X_0=p.1_\F=1_\F$$
because $p$ is odd. Besides, a direct computation yields
$$SX_0=\begin{bmatrix}
\sqrt{a_1} & \cdots & \sqrt{a_p} & [0]_{q \times 1} & \sqrt{a_{p+1}} & \cdots & \sqrt{a_{p+q}} & [0]_{2m \times 1}
\end{bmatrix}^T=\Delta(S)^{1/2},$$
and hence
$$S_a X_0=SX_0+\Bigl(\bigl(\Delta(S)^{1/2}\bigr)^T X_0\Bigr)\,\Delta(S)^{1/2}=2 \Delta(S)^{1/2}=0.$$
Since $X_0 \neq 0$, we deduce that $\Ker S_a=\overline{\F} X_0$ and hence $X_0 \in \overline{W_1} \cap \overline{W_2}$.

Now, consider an $\F$-linear hyperplane $H$ of $\F^n$ that includes $W_1$, with equation
$\underset{i=1}{\overset{n}{\sum}} \lambda_i\, x_i=0$
where the $\lambda_i$'s all belong to $\F$.
This equation must also be satisfied by all the vectors of $\overline{W_1}$.
In particular, as $X_0 \in \overline{W_1}$ we have
$$\sum_{i=1}^p (\lambda_i (a_i)^{-1})\,\sqrt{a_i}+\sum_{i=p+1}^{p+q} \lambda_i\,\sqrt{a_i}=0.$$
Since the $\sqrt{a_i}'s$ are linearly independent over $\F$, it follows that $\lambda_i=0$ for all $i \in \lcro 1,p+q\rcro$.
Since this holds for every linear hyperplane of $\F^n$ that includes $W_1$,
we obtain by duality that $\F^{p+q} \times \{0\} \subset W_1$.
As $(\Ker Q)^\bot=\Vect(e_1,\dots,e_{p+q})$, we deduce that
$(\Ker Q)^\bot \subset V_1$. Likewise, $(\Ker Q)^\bot \subset V_2$.

We are about to conclude. What we have just seen shows that $(\Ker Q)^\bot$ is $b_a$-orthogonal to itself.
Setting $D:=\Diag(a_1,\dots,a_p)$, it follows that $0=D_a=D+\Delta(D)^{1/2} (\Delta(D)^{1/2})^T$,
whence $D=\Delta(D)^{1/2} (\Delta(D)^{1/2})^T$. Since $p=\rk D$ is odd, we conclude that $p=1$.
Finally, the Witt index of $b$ equals $m+q=\frac{n-p}{2}=\frac{n-1}{2}$, and the proof is complete.
\end{proof}

\subsection{Two orthogonality lemmas on nilpotent spaces of $b$-symmetric endomorphisms}

We start by recalling a very classical lemma on nilpotent spaces of endomorphisms
(see, e.g., section 2.1 of \cite{dSPStructured1} for a proof):

\begin{lemma}[Basic orthogonality lemma]
Let $\calV$ be a nilpotent linear subspace of $\End(V)$. Then
$\tr(u \circ v)=0$ for all $u$ and $v$ in $\calV$.
\end{lemma}

Here is a simple consequence:

\begin{lemma}[First orthogonality lemma for tensors]\label{tensorortho1}
Let $\calV$ be a nilpotent linear subspace of $\calS_b$.
Let $x \in V\setminus \{0\}$ be a non-zero isotropic vector. Let $u \in \calV$, and
let $(\alpha,y) \in \F \times V$ be such that $\calV$ contains $\alpha\, x \otimes_b x+ x \wedge_b y$.
Then, $\alpha\, b(x,u(x))=0$.
\end{lemma}

\begin{proof}
Set $v:=\alpha\, x \otimes_b x+ x \wedge_b y$.
By the basic orthogonality lemma, the endomorphism
$$u \circ v=\alpha\, b(x,-) \otimes u(x)+b(y,-) \otimes u(x)+b(x,-) \otimes u(y)$$
has trace zero, and the claimed result ensues by noting that the trace of $u \circ v$ equals
$$\alpha\, b(x,u(x))+b(y,u(x))+b(x,u(y))=\alpha\, b(x,u(x))+2\,b(y,u(x))=\alpha\, b(x,u(x)).$$
\end{proof}

The second result is slightly different as it is related to $b_a$-orthogonality rather than to $b$-orthogonality.

\begin{lemma}[Second orthogonality lemma for tensors]\label{tensorortho2}
Let $\calV$ be a nilpotent linear subspace of $\calS_b$.
Let $x \in V\setminus \{0\}$ be a non-zero isotropic vector. Let $u \in \calV$ be such that $b(x,u(x))=0$, and
let $y \in \{x\}^\bot \setminus \F x$ be such that $\calV$ contains $x \wedge_b y$.
Then, $b_a(u(x),y)=0$.
\end{lemma}

\begin{proof}
Set $n:=\dim V-2$.
We choose a vector $x' \in V$ such that $b(x,x')=1$. Since $b(x,x)=0$, we see that $x$ and $x'$ are linearly independent.
Set $c:=b(x',x')$.

With respect to the basis $(x,x')$, the matrix of the restriction of $b$ to $\Vect(x,x')^2$ equals
$\begin{bmatrix}
0 & 1 \\
1 & c
\end{bmatrix}$. This matrix is invertible, and hence this restriction is non-degenerate.
It follows that $V=\Vect(x,x') \botoplus \Vect(x,x')^\bot$ and that the
restriction of $b$ to $(\Vect(x,x')^\bot)^2$ is non-degenerate.
Now, we consider a basis $(f_1,\dots,f_n)$ of $\Vect(x,x')^\bot$, and the matrix $P:=(b(f_i,f_j))_{1 \leq i,j \leq n}$.
The family $\bfC:=(x,f_1,\dots,f_n,x')$ is a basis of $V$, and the matrix $S:=\Mat_\bfC(b)$ satisfies
$$S=
\begin{bmatrix}
0 & [0]_{1 \times n} & 1 \\
[0]_{n \times 1} & P & [0]_{n \times 1} \\
1 & [0]_{1 \times n} & c
\end{bmatrix}.$$

One checks that the $S$-symmetric matrices are exactly the matrices of the form
$$L=\begin{bmatrix}
\lambda(L) & (PC_2(L))^T+c\,(PC_1(L))^T & \mu(L) \\
C_1(L) & P^{-1}K(L) & C_2(L) \\
\delta(L) & (PC_1(L))^T & \lambda(L)+c\, \delta(L)
\end{bmatrix}$$
with $\lambda(L)$, $\mu(L)$ and $\delta(L)$ in $\F$, $C_1(L)$ and $C_2(L)$ in $\F^n$, and $K(L)$ in $\Mats_n(\F)$.

Now, we consider the respective matrices $M$ and $N$ of $u$ and $v:=x \wedge_b y$
with respect to $\bfC$. Since $u(x) \in \{x\}^\bot$, we actually have $\delta(M)=0$, so that
$$M=\begin{bmatrix}
\lambda(M) & (PC_2(M))^T+c\,(PC_1(M))^T & \mu(M) \\
C_1(M) & P^{-1}K(M) & C_2(M) \\
0 & (PC_1(M))^T & \lambda(M)
\end{bmatrix}.$$
Besides, $v$ maps $\{x\}^\bot$ into $\F x$, and hence $\delta(N)=0$, $C_1(N)=0$ and $K(N)=0$; moreover, $v(x)=0$ and hence $\lambda(N)=0$. Finally,
we have $y=\alpha x+z$ for some $\alpha \in \F$ and some $z \in \Vect(x,x')^\bot$, and one computes that
$$v(x')=b(y,x')\, x-b(x,x')\, y=\alpha x-y=-z=z.$$
Hence, $\mu(N)=0$.
Let us set $X:=C_1(M)$ and $Y:=C_2(N)$.
Thus,
$$N=\begin{bmatrix}
0 & (PY)^T & 0 \\
[0]_{n \times 1} & 0_n & Y \\
0 & [0]_{1 \times n} & 0
\end{bmatrix}.$$

Next, the matrices $M$ and $M+N$ represent $u$ and $u+v$, respectively, and since $u$ and $u+v$ belong to $\calV$
we deduce that $M$ and $M+N$ are nilpotent.
We see that
$$t I_{n+2}+M=\begin{bmatrix}
t+\lambda(M) & (PC_2(M))^T+c\,(PX)^T & \mu(M) \\
X & t I_n+P^{-1}K(M) & C_2(M) \\
0 & (PX)^T & t+\lambda(M)
\end{bmatrix}$$
and
$$t I_{n+2}+M+N=\begin{bmatrix}
t+\lambda(M) & (PC_2(M))^T+c\,(PX)^T+(PY)^T & \mu(M) \\
X & t I_n+P^{-1}K(M) & C_2(M)+Y \\
0 & (PX)^T & t+\lambda(M)
\end{bmatrix}.$$
Using the linearity of the determinant with respect to the first row, we obtain
\begin{multline*}\det(t I_{n+2}+M+N)=\begin{vmatrix}
t+\lambda(M) & (PC_2(M))^T+c(PX)^T & \mu(M) \\
X & t I_n+P^{-1}K(M) & C_2(M)+Y \\
0 & (PX)^T & t+\lambda(M)
\end{vmatrix} \\
+\begin{vmatrix}
0 & (PY)^T & 0 \\
X & t I_n+P^{-1}K(M) & C_2(M)+Y \\
0 & (PX)^T & t+\lambda(M)
\end{vmatrix}.
\end{multline*}
Using the linearity of the determinant with respect to the last column, we deduce that
$$\det(tI_{n+2}+M+N)=\det(tI_{n+2}+M)+B_1(t)+B_2(t)$$
where
$$B_1(t):=\begin{vmatrix}
t+\lambda(M) & (PC_2(M))^T+c\,(PX)^T & 0 \\
X & t I_n+P^{-1}K(M) & Y \\
0 & (PX)^T & 0
\end{vmatrix}+\begin{vmatrix}
0 & (PY)^T & 0 \\
X & tI_n+P^{-1}K(M) & C_2(M) \\
0 & (PX)^T & t+\lambda(M)
\end{vmatrix}$$
and
$$B_2(t):=\begin{vmatrix}
0 & (PY)^T & 0 \\
X & tI_n+P^{-1}K(M) & Y \\
0 & (PX)^T & 0
\end{vmatrix}.$$

As $M+N$ and $M$ are nilpotent, we have $\det(t I_{n+2}+M+N)=t^{n+2}=\det(tI_{n+2}+M)$, and it ensues that $B_1(t)=B_2(t)$.

Next, we prove that $B_1(t)=0$.
To see this, note that the matrix in the first summand of the definition of $B_1(t)$ is closely related to the
$S$-adjoint of the matrix in the second one!
Precisely, if we set
$$A_1:=\begin{bmatrix}
0 & (PY)^T & 0 \\
X & tI_n+P^{-1}K(M) & C_2(M) \\
0 & (PX)^T & t+\lambda(M)
\end{bmatrix},$$
then a tedious but straightforward computation reveals that
$$(SA_1 S^{-1})^T=\begin{bmatrix}
t+\lambda(M) & (PC_2(M))^T+c\,(PX)^T & c(t+\lambda(M)) \\
X & t I_n+P^{-1}K(M) & Y+c\, X \\
0 & (PX)^T & 0
\end{bmatrix}$$
(this uses the observation that $S^{-1}=\begin{bmatrix}
c & [0]_{1 \times n} & 1 \\
[0]_{n \times 1} & P^{-1} & [0]_{n \times 1} \\
1 & [0]_{1 \times n} & 0
\end{bmatrix}$ and that $P$ and $K(M)$ are symmetric).

Using the column operation $C_{n+2} \leftarrow C_{n+2}-c\,C_1$ leads to
$$\det \bigl((SA_1 S^{-1})^T\bigr)=\begin{vmatrix}
t+\lambda(M) & (PC_2(M))^T+c\,(PX)^T & 0 \\
X & t I_n+P^{-1}K(M) & Y \\
0 & (PX)^T & 0
\end{vmatrix}$$
and it follows that $B_1(t)=0$.

We conclude that $B_2(t)=0$. In other words, the matrix
$$\begin{bmatrix}
0 & (PY)^T & 0 \\
X & tI_n+P^{-1}K(M) & Y \\
0 & (PX)^T & 0
\end{bmatrix}.$$
is singular. Yet, $tI_n+P^{-1}K(M)$ is invertible as a matrix with entries in $\F(t)$ (because its determinant is the characteristic polynomial of $P^{-1}K(M)$), so by Gaussian elimination we deduce that
the $2$ by $2$ matrix
$$\begin{bmatrix}
(PX)^T \\
(PY)^T
\end{bmatrix} (tI_n+P^{-1}K(M))^{-1} \begin{bmatrix}
X & Y
\end{bmatrix}$$
is singular. Hence, by factoring $(tI_n+P^{-1}K(M))^{-1}$ by $t^{-1}$ and then multiplying with $\det(I_n+t^{-1} P^{-1}K(M))$ we obtain that
the matrix
$$\begin{bmatrix}
(PX)^T \\
(PY)^T
\end{bmatrix} (I_n+t^{-1} P^{-1}K(M))^\ad \begin{bmatrix}
X & Y
\end{bmatrix}$$
is also singular.
It follows that the polynomial mapping
$$f : u \in \overline{\F} \mapsto \det\left(\begin{bmatrix}
(PX)^T \\
(PY)^T
\end{bmatrix} \bigl(I_n+u P^{-1}K(M)\bigr)^\ad \begin{bmatrix}
X & Y
\end{bmatrix}\right)$$
vanishes at every non-zero element of $\overline{\F}$. Since $\overline{\F}$ is infinite
we deduce that $f$ is identically zero and in particular $f(0)=0$.
Hence the determinant of $\begin{bmatrix}
(PX)^T \\
(PY)^T
\end{bmatrix}\begin{bmatrix}
X & Y
\end{bmatrix}$ equals zero, which reads as follows because $P$ is symmetric:
$$\bigl(X^TPX\bigr)\bigl(Y^TPY\bigr)
=\bigl(X^TPY\bigr)^2.$$
We recall that $\Delta(P)^{1/2}$ denotes the column whose entries are the square roots in $\overline{\F}$ of the diagonal entries of $P$.

Since $\F$ has characteristic $2$ the above identity can be rephrased as follows:
$$(X^T \Delta(P)^{1/2})^2 ((\Delta(P)^{1/2})^T Y)^2=\bigl(X^TPY\bigr)^2$$
and
$$X^T \Delta(P)^{1/2} (\Delta(P)^{1/2})^T Y=X^TPY$$
and finally
$$X^T P_a Y=0.$$
We are about to conclude. Since both $u(x)$ and $y$ belong to $\{x\}^\bot$, we can split them
into $u(x)=\alpha_1\, x+z_1$ and $y=\alpha_2\, x+z_2$ for some $(\alpha_1,\alpha_2) \in \F^2$ and vectors $z_1,z_2$ of $\Vect(x,x')^\bot$.
The vectors $z_1$ and $z_2$ are represented in the basis $(f_1,\dots,f_n)$ by $X$ and $Y$, respectively. Hence,
the equality $X^T P_a Y=0$ means that $b_a(z_1,z_2)=0$. Since $b_a$ is alternating, we conclude that
$$b_a(u(x),y)=\alpha_1\, b_a(x,z_2)+\alpha_2\, b_a(x,z_1)+b_a(z_1,z_2)=\alpha_1\, b(x,z_2)+\alpha_2\, b(x,z_1)=0,$$
where we have used the fact that $Q(x)=0$ to obtain $b_a(x,z_1)=b(x,z_1)$ and $b_a(x,z_2)=b(x,z_2)$.
\end{proof}

\section{An upper-bound for the dimension: the special case when $\Ker Q$ is totally $b$-singular}\label{m=0Section}

Here, we prove points (b) and (c) of Theorem \ref{maintheo} in the special case when $\Ker Q$ is totally $b$-singular,
so that $\dim \Ker Q=\nu(b)$.

\begin{lemma}\label{SKer=Kerlemma}
Let $b$ be a non-degenerate symmetric bilinear form on a finite-dimensional vector space $V$, with attached quadratic form $Q$.
Assume that $\Ker Q$ is totally $b$-singular. Then $\Ker Q$ is stable under every nilpotent element of $\calS_b$.
\end{lemma}

\begin{proof}
Let $u \in \calS_b$ be nilpotent.
By Lemma \ref{stablesetimlemma}, there is a totally $b$-singular subspace $F$ of $V$ with dimension $\nu(b)$
that is stable under $u$. Then, $F \subset \Ker Q$.
Since $\Ker Q$ is totally $b$-singular, the definition of $\nu(b)$ shows that $F=\Ker Q$, and the claimed result ensues.
\end{proof}

\begin{prop}\label{SKer=KerProp}
Let $b$ be a non-degenerate symmetric bilinear form on a finite-dimensional vector space $V$, with attached quadratic form $Q$.
Assume that $\Ker Q$ is totally $b$-singular. Set $\nu:=\dim \Ker Q$.
Let $\calV$ be a nilpotent linear subspace of $\calS_b$ (respectively, of $\calA_b$).
Then $\dim \calV \leq \nu(n-\nu)$ (respectively, $\dim \calV \leq \nu(n-\nu-1)$).
\end{prop}

\begin{proof}
Let $u \in \calV$. Since $u$ is nilpotent, Lemma \ref{stablesetimlemma} shows that it stabilizes $\Ker Q$, and hence
Lemma \ref{lemmakerim} shows that $u$ also stabilizes $(\Ker Q)^\bot$. Since $\Ker Q \subset (\Ker Q)^\bot$,
the symmetric bilinear form $b$ induces a symmetric bilinear form $\overline{b}$ on the quotient space $E:=(\Ker Q)^\bot/\Ker Q$,
and obviously $\overline{b}$ is non-isotropic. Moreover, $u$ induces a $\overline{b}$-symmetric endomorphism $\overline{u}$ of $E$.
By Lemma \ref{lemmanonisotropic}, we find $\overline{u}=0$, i.e.\ $u$ maps $(\Ker Q)^\bot$ into $\Ker Q$.

Now, let us choose a basis $(e_1,\dots,e_\nu)$ of $\Ker Q$, together with a subspace $H$ which is complementary to $(\Ker Q)^\bot$ in $V$.
Since $b$ is non-degenerate, $\dim H=\dim \Ker Q=\nu$
and we may recover a basis $(f_1,\dots,f_\nu)$ of $H$ such that $b(e_i,f_j)=\delta_{i,j}$ for all $(i,j) \in \lcro 1,\nu\rcro^2$.
Then, $b$ induces a non-degenerate symmetric bilinear form on $\Ker Q \oplus H$, so that
$V=(\Ker Q \oplus H) \botoplus (\Ker Q \oplus H)^\bot$. Let us choose a basis
$(g_1,\dots,g_{n-2\nu})$ of $(\Ker Q \oplus H)^\bot$.
Then, $\bfB:=(e_1,\dots,e_\nu,g_1,\dots,g_{n-2\nu},f_1,\dots,f_\nu)$ is a basis of $V$, and the matrix $S$ of
$b$ with respect to that basis looks as follows:
$$S=\begin{bmatrix}
0_\nu & [0]_{\nu \times (n-2\nu)} & I_\nu \\
[0]_{(n-2\nu) \times \nu} & P & [0]_{(n-2\nu) \times \nu} \\
I_\nu & [0]_{\nu \times (n-2\nu)} & D
\end{bmatrix}$$
for some $P \in \GL_{n-2\nu}(\F) \cap \Mats_{n-2\nu}(\F)$ and some symmetric matrix $D \in \Mats_\nu(\F)$.

Denote by $\calM$ the space of all matrices that represent the elements of $\calV$ with respect to $\bfB$.
Since every element of $\calV$ maps $(\Ker Q)^\bot$ into $\Ker Q$, every element $M$ of $\calM$ has the following form
$$M=\begin{bmatrix}
A(M) & B(M) & E(M) \\
[0]_{(n-2\nu) \times \nu} & 0_{n-2\nu} & C(M) \\
0_\nu & [0]_{\nu \times (n-2\nu)} & F(M)
\end{bmatrix}.$$
Moreover, given such a matrix $M$, the fact that $M$ is $S$-symmetric (respectively, $S$-alternating)
yields that $F(M)=A(M)^T$, $B(M)=(PC(M))^T$ and $E(M)+DF(M)$ is symmetric (respectively, alternating).
Hence, we find linear mappings $A : \calM \rightarrow \Mat_\nu(\F)$, $C : \calM \rightarrow \Mat_{n-2\nu,\nu}(\F)$
and $G : \calM \rightarrow \Mats_\nu(\F)$ (respectively, $G : \calM \rightarrow \Mata_\nu(\F)$)
such that every matrix $M$ of $\calM$ has the form
$$M=\begin{bmatrix}
A(M) & (P C(M))^T & D A(M)^T+G(M) \\
[0]_{(n-2\nu) \times \nu} & 0_{n-2\nu} & C(M) \\
0_\nu & [0]_{\nu \times (n-2\nu)} & A(M)^T
\end{bmatrix}.$$
We deduce the injectivity of the mapping
$$M \in \calM \mapsto (A(M),C(M),G(M)) \in A(\calM) \times \Mat_{n-2\nu,\nu}(\F) \times \Mats_\nu(\F)$$
(respectively, of the mapping
$$M \in \calM \mapsto (A(M),C(M),G(M)) \in A(\calM) \times \Mat_{n-2\nu,\nu}(\F) \times \Mata_\nu(\F)\; \bigr).$$
Moreover, for all $M \in \calM$, since $M$ is nilpotent and block-upper-triangular we obtain that $A(M)$ is nilpotent. It follows that
$A(\calM)$ is a nilpotent linear subspace of $\Mat_\nu(\F)$. Hence, the standard Gerstenhaber theorem yields
$$\dim A(\calM) \leq \dbinom{\nu}{2}.$$
It follows that
$$\dim \calV=\dim \calM \leq \dim \calA(\calM)+\dim \Mat_{n-2\nu,\nu}(\F)+\dim \Mats_\nu(\F)
=\nu^2+(n-2\nu)\nu=\nu(n-\nu)$$
(respectively,
$$\dim \calV=\dim \calM \leq \dim \calA(\calM)+\dim \Mat_{n-2\nu,\nu}(\F)+\dim \Mata_\nu(\F)
=\nu^2-\nu+(n-2\nu)\nu=\nu(n-\nu-1).)$$
\end{proof}

\section{Nilpotent subspaces of $b$-alternating endomorphisms}\label{proofalternatingSection}

Here, we prove points (b) and (c) of Theorem \ref{maintheo}. The proof is by induction on the dimension of $V$,
and its main key is the second orthogonality lemma for tensors (Lemma \ref{tensorortho2}).

So, let $\F$ be a field with characteristic $2$, let $V$ be a finite-dimensional vector space over $\F$,
and let $b$ be a non-degenerate symmetric bilinear form on $V$, whose associated quadratic form we denote by $Q$. Let $\calV$ be a nilpotent linear subspace of $\calA_b$.

Set $n:=\dim V$ and $\nu:=\nu(b)$. If $\nu=0$ then $\calV=\{0\}$ by Lemma \ref{lemmanonisotropic}, and hence $\dim \calV \leq \nu(n-\nu-1)$.
If $n=2\nu+1$ and $\Ker Q$ is totally $b$-singular, then $\SKer Q=\Ker Q$ has dimension $\nu$ and hence
$\nu(n-\nu)-\dim \SKer Q=\nu(n-\nu-1)$. Hence, in that case the result follows directly from Proposition \ref{SKer=KerProp}.

In the remainder of the proof, we discard the special case when $n=2\nu+1$ and $\Ker Q$ is totally $b$-singular, and we assume that $\nu>0$.
Since $\nu>0$ we can choose $x \in V \setminus \{0\}$ such that $Q(x)=0$.
Set
$$\calV x:=\{u(x) \mid u \in \calV\}.$$
We consider the kernel
$$\calU_{\calV,x}:=\{u \in \calV : u(x)=0\}$$
of the surjective linear mapping
$$u \in \calV \mapsto u(x) \in \calV x.$$
Each $u \in \calU_{\calV,x}$ stabilizes $\{x\}^\bot$ (because it stabilizes $\F x$)
and hence induces a nilpotent endomorphism $\overline{u}$ of the quotient space $\{x\}^\bot/\F x$.
Note that $b$ induces a non-degenerate symmetric bilinear form $\overline{b}$ on $\{x\}^\bot/\F x$.
By Lemma \ref{quotientindexlemma}, the Witt index of $\overline{b}$ equals $\nu-1$.
It is then easily checked that, for every $u \in \calU_{\calV,x}$, the endomorphism $\overline{u}$ is still $\overline{b}$-alternating.
Hence,
$$\calV \modu x:=\{\overline{u} \mid u \in \calU_{\calV,x}\}$$
is a nilpotent linear subspace of $\calA_{\overline{b}}$.
Finally, the kernel of the linear mapping
$$\varphi : u \in \calU_{\calV,x} \mapsto \overline{u} \in \calV \modu x$$
consists of the operators $u \in \calV$ that vanish at $x$ and map $\{x\}^\bot$ into $\F x$.
Set
$$L_{\calV,x}:=\bigl\{y \in V : \; x \wedge_b y \in \calV\bigr\}$$
By Proposition \ref{caracalternatingtensors}, $L_{\calV,x}$ is a linear subspace of $\{x\}^\bot$
that contains $x$, the kernel of $\varphi$ equals $x \wedge_b L_{\calV,x}$ and
$$\dim \Ker \varphi=\dim (L_{\calV,x})-1.$$

By applying the rank theorem twice, we obtain
$$\dim \calV=\dim (\calV x)+\dim L_{\calV,x}-1+\dim (\calV \modu x).$$
Since $\calV$ is a space of $b$-alternating endomorphisms, we have $\calV x \subset \{x\}^\bot$.
Hence, by the second orthogonality lemma for tensors, we obtain that
$\calV x$ is $b_a$-orthogonal to $L_{\calV,x}$. Now,
denote by
$$\pi : \{x\}^\bot \rightarrow \{x\}^\bot/\F x$$
the canonical projection.
Note that $b_a(x,z)=b(x,z)+\sqrt{Q(x)Q(z)}=0+0=0$ for all $z \in \{x\}^\bot$.
It is clear from there that $b_a$ induces a symmetric bilinear form on $\{x\}^\bot/\F x$,
and this form is no other than the a-transform of $\overline{b}$!
Since $\calV x$ is $b_a$-orthogonal to $L_{\calV,x}$, we obtain that
$\pi(\calV x)$ is $\overline{b}_a$-orthogonal to $\pi(L_{\calV,x})$.

As $x \in L_{\calV,x}$, we have
$$\dim \bigl(\pi(L_{\calV,x})\bigr)=\dim L_{\calV,x}-1.$$
Besides, $x \not\in \calV x$ because every element of $\calV$ is nilpotent,
and hence
$$\dim \bigl(\pi(\calV x)\bigr)=\dim (\calV x).$$
Therefore,
$$\dim \calV=\dim \bigl(\pi(\calV x)\bigr)+\dim(\pi(L_{\calV,x}))+\dim (\calV \modu x).$$
From here, we split the discussion into two cases.

\vskip 3mm
\noindent \textbf{Case 1: $n \neq 2\nu+1$.} \\
Then $n-2 \neq 2(\nu-1)+1$. Hence, on the one hand the induction hypothesis yields
$$\dim (\calV \modu x) \leq (\nu-1)\bigl((n-2)-(\nu-1)-1\bigr),$$
and on the other hand Lemma \ref{atransformlemma} yields
$$\dim \bigl(\pi(\calV x)\bigr)+\dim(\pi(L_{\calV,x})) \leq n-2.$$
Hence,
$$\dim \calV \leq (\nu-1)\bigl((n-2)-(\nu-1)-1\bigr)+n-2=\nu(n-\nu-1).$$

\vskip 3mm
\noindent \textbf{Case 2: $n=2\nu+1$.} \\
Then, $\Ker Q \not\subset (\Ker Q)^\bot$ because of our starting assumptions.
In that case, we choose $x$ in $\Ker Q \setminus (\Ker Q)^\bot$.

Denote by $\overline{Q}$ the quadratic form associated with $\overline{b}$.
Because $x \in \Ker Q$, we have $\SKer Q \subset \{x\}^\bot$. We claim that
$\pi(\SKer Q) \subset \SKer \overline{Q}$
(the converse inclusion also holds but we will not need it).
Let $z \in \SKer Q$. Firstly, $\overline{Q}(\pi(z))=Q(z)=0$.
Next, let $s\in \Ker \overline{Q}$, so that $s=\pi(z')$ for some $z' \in \Ker Q$.
Then, $b(z,z')=0$ leads to $\overline{b}(\pi(z),\pi(z'))=0$, and hence $\pi(\SKer Q)$ is $\overline{b}$-orthogonal to $\Ker \overline{Q}$.
This yields the claimed inclusion
$$\pi(\SKer Q) \subset \SKer \overline{Q}.$$

Moreover, since $x \not\in \SKer Q$ we have $\dim (\pi(\SKer Q))=\dim \SKer Q$ and hence
$$\dim (\SKer \overline{Q}) \geq \dim (\SKer Q).$$

Now, as $\dim(\{x\}^\bot/\F x)=n-2=2(\nu-1)+1$ and $\overline{b}$ has Witt index $\nu-1$, we have by induction
$$\dim (\calV \modu x) \leq (\nu-1)\bigl((n-2)-(\nu-1)\bigr)-\dim (\SKer \overline{Q})
\leq (\nu-1)(n-\nu-1)-\dim(\SKer Q).$$
Moreover, the third conclusion in Lemma \ref{atransformlemma} shows that
$$\dim  \bigl(\pi(\calV x)\bigr)+\dim\bigl(\pi(L_{\calV,x})\bigr) \leq n-1.$$
Hence,
$$\dim \calV \leq (\nu-1)(n-\nu-1)+n-1-\dim(\SKer Q)=\nu(n-\nu)-\dim(\SKer Q).$$
This completes the proof.

\section{Nilpotent subspaces of $b$-symmetric endomorphisms}\label{proofsymmetricSection}

Here, we complete our study by proving point (a) of Theorem \ref{maintheo}.
The proof strategy is quite similar to the one of points (b) and (c) of the same theorem, with additional technicalities
however. Remember that $\F$ is assumed to have characteristic $2$. Before we can prove the result, we need
a small trick that allows us to reduce the situation where $n=2\nu(b)+1$ to the convenient one where, in addition, $\F$ is perfect
(i.e.\ $x \in \F \mapsto x^2 \in \F$ is surjective).
This is explained in the next paragraph. The proof \emph{per se} is carried out afterwards.

\subsection{Extending scalars}

\begin{lemma}\label{perfectoddlemma}
Let $b$ be a non-degenerate symmetric bilinear form on a finite-dimensional vector space $V$.
Assume that $\F$ is perfect and that $\dim V$ is odd. Then, the Witt index of $b$ equals $\nu:=\frac{\dim V-1}{2}$, and
$b$ is represented in some basis by
$I_1 \oplus \begin{bmatrix}
0_\nu & I_\nu \\
I_\nu & 0_\nu
\end{bmatrix}$.
\end{lemma}

\begin{proof}
Let us take a normal basis $(e_1,\dots,e_n)$ for $b$, with associated indices $p,q,m$.
Since $n=\dim V$ is odd we get that $p$ is odd, and in particular $p>0$. Since $\F$ is perfect,
we have $\sqrt{x} \in \F$ for all $x \in \F$, and hence for every family $(b_1,\dots,b_N)$ of elements of $\F$
for which $\sqrt{b_1},\dots,\sqrt{b_N}$ are linearly independent over $\F$, we must have $N \leq 1$.
It follows that $p+q \leq 1$, leading to $p=1$ and $q=0$, and hence $m=\frac{n-1}{2}=\nu$. Hence, the Witt index of $b$ equals $q+m=\nu$.
Finally, the basis $(\sqrt{b(e_1,e_1)}^{-1} e_1,e_2,\dots,e_n)$ satisfies the claimed property.
\end{proof}

\begin{lemma}\label{perfectlemma}
Let $b$ be a non-degenerate symmetric bilinear form on a finite-dimensional vector space $V$ over $\F$.
Let $\calV$ be a nilpotent linear subspace of $\calS_b$. Then, there exist a perfect field $\F'$
with characteristic $2$, a vector space $V'$ over $\F'$ such that $\dim_\F V=\dim_{\F'} V'$,
a non-degenerate symmetric bilinear form $b'$ on $V'$, and
a nilpotent linear subspace $\calV'$ of $\calS_{b'}$ such that $\dim_\F \calV=\dim_{\F'} \calV'$.
\end{lemma}

\begin{proof}
If $\F$ is finite then it is perfect and we simply take $(\F',V',b',\calV'):=(\F,V,b,\calV)$.

Assume now that $\F$ is infinite. Then, we take $\F':=\overline{\F}$.
We will use the matrix viewpoint. Set $n:=\dim V$ and $d:=\dim \calV$.
Choose a basis $\bfB$ of $V$. Denote by $S$ the matrix of $b$ with respect to $\bfB$, and by $\calM$
the $\F$-linear subspace of $\Mat_n(\F)$ associated with $\calV$ with respect to the basis $\calM$. It has dimension $d$.
Set $\calM':=\Vect_{\F'}(\calM)$, which is an $\F'$-linear subspace of $\Mat_n(\F')$
with dimension $d$ over $\F'$.
Since $S \calM \subset \Mats_n(\F)$, we have $S \calM' \subset \Mats_n(\F')$.
Finally, consider the space $V':=(\F')^n$ equipped with the non-degenerate symmetric $\F'$-bilinear form
$b'$ whose matrix with respect to the standard basis of $V'$ is $S$. Denote by $\calV'$ the set of all endomorphisms of $V'$
represented by the elements of $\calM'$ with respect to the standard basis. Then, $\calV' \subset \calS_{b'}$ and
$$\dim_{\F'} \calV'=\dim_{\F'} \calM'=\dim_\F \calM=\dim_\F \calV.$$
In order to conclude, it remains to prove that all the elements of $\calV'$ are nilpotent.
We know that every element of $\calM$ is nilpotent, and it suffices to prove that so is every element of $\calM'$.
The line of reasoning here is classical: let $(M_1,\dots,M_p)$ be an arbitrary list of elements of $\calM$.
The mapping
$$f : (x_1,\dots,x_p) \in (\F')^p \mapsto \biggl(\sum_{k=1}^p x_k M_k\biggr)^n$$
is a (vector-valued) polynomial function, and by assumption it vanishes everywhere on $\F^p$. Since $\F$ is infinite,
we deduce that $f$ is identically zero. It ensues that every element of $\calM'$ is nilpotent, and we conclude that
$\calV'$ is nilpotent.
\end{proof}

\subsection{The inductive proof}

Let $V$ be a finite-dimensional vector space over $\F$ and $b$ be a non-degenerate symmetric bilinear form on $V$ with attached quadratic form $Q$. Let $\calV$ be a nilpotent linear subspace of $\calS_b$. Set $n:=\dim V$ and $\nu:=\nu(b)$. If $\nu=0$ then $\calV=\{0\}$ by Lemma \ref{lemmanonisotropic}, and hence $\dim \calV \leq \nu(n-\nu)$.

Now, we assume that $\nu>0$. We use a \emph{reductio ad absurdum}: we assume that $\dim \calV>\nu(n-\nu)$, and we set out to find a contradiction.

Here is the first major result we wish to prove:

\begin{claim}\label{claim1}
One has $n=2\nu+1$. Moreover, $\calV$ contains  $x \wedge_b y$ for all $x \in \Ker Q$ and all $y \in (\Ker Q)^\bot$.
\end{claim}

\begin{proof}
Let $x \in V$ be isotropic and non-zero (such a vector actually exists because $\nu>0$).
Set
$$\calV x:=\{u(x) \mid u \in \calV\}.$$
We consider the kernel
$$\calU_{\calV,x}:=\{u \in \calV : u(x)=0\}$$
of the surjective linear mapping
$$u \in \calV \mapsto u(x) \in \calV x.$$
Each $u \in \calU_{\calV,x}$ stabilizes $\{x\}^\bot$ (because it stabilizes $\F x$)
and hence induces a nilpotent endomorphism $\overline{u}$ of the quotient space $\{x\}^\bot/\F x$.
Note that $b$ induces a non-degenerate symmetric bilinear form $\overline{b}$ on $\{x\}^\bot/\F x$
with Witt index $\nu-1$. We denote by $\overline{Q}$ the associated quadratic form.
For every $u \in \calU_{\calV,x}$, the endomorphism $\overline{u}$ is $\overline{b}$-symmetric.
Hence,
$$\calV \modu x:=\{\overline{u} \mid u \in \calU_{\calV,x}\}$$
is a nilpotent linear subspace of $\calS_{\overline{b}}$.
Finally, the kernel of the linear mapping
$$\varphi : u \in \calU_{\calV,x} \mapsto \overline{u} \in \calV \modu x$$
consists of the operators $u \in \calV$ that vanish at $x$ and map $\{x\}^\bot$ into $\F x$.
Set
$$L_{\calV,x}:=\bigl\{y \in V : \; x \wedge_b y \in \calV\bigr\}$$
and
$$L'_{\calV,x}:=\{(\alpha,y) \in \F \times V : \; \alpha\, x\otimes_b x+x \wedge_b y \in \calV\bigr\}.$$
By Proposition \ref{caracsymmetrictensors}, $L_{\calV,x}$ is a linear subspace of $\{x\}^\bot$
that contains $x$ and
$$\Ker \varphi=\bigl\{\alpha\, x\otimes_b x+x \wedge_b y \mid (\alpha,y) \in L'_{\calV,x} \bigr\}.$$
Moreover, we have seen that the kernel of $(\alpha,y) \in \F \times \{x\}^\bot \mapsto  \alpha\, x\otimes_b x+x \wedge_b y$
equals $\{0\} \times \F x$, which leads to
$$\dim(\Ker \varphi)=\dim L'_{\calV,x}-1.$$

By applying the rank theorem twice, we obtain
$$\dim \calV=\dim (\calV x)+\dim L'_{\calV,x}-1+\dim (\calV \modu x).$$
Next, we claim that
\begin{equation}\label{dimensioninequality}
\dim (\calV x)+\dim L'_{\calV,x} \leq \dim (\calV x \cap \{x\}^\bot)+\dim L_{\calV,x}+1.
\end{equation}
To see this, note first that the rank theorem yields
$\dim (\calV x)\leq \dim (\calV x\cap \{x\}^\bot)+1$ and
$\dim L'_{\calV,x} \leq \dim L_{\calV,x}+1$, and hence
$$\dim (\calV x)+\dim L'_{\calV,x} \leq  \dim (\calV x \cap \{x\}^\bot)+\dim L_{\calV,x}+2.$$
Moreover, if equality holds then $\dim (\calV x)= \dim (\calV x\cap \{x\}^\bot)+1$ and
$\dim L'_{\calV,x}=\dim L_{\calV,x}+1$, which yields an operator $u \in \calV$ such that $u(x)\not\in \{x\}^\bot$
and a non-zero scalar $\alpha$ and a vector $y \in \{x\}^\bot$ such that $\alpha\, x\otimes_b x+x \wedge_b y \in \calV$;
then the first orthogonality lemma for tensors shows that $\alpha\, b(x,u(x)) = 0$, which is contradictory.
This yields inequality \eqref{dimensioninequality}, and it ensues that
$$\dim \calV \leq \dim (\calV x \cap \{x\}^\bot)+\dim L_{\calV,x}+\dim (\calV \modu x).$$
Next, by the second orthogonality lemma for tensors, the space
$\calV x \cap \{x\}^\bot$ is $b_a$-orthogonal to $L_{\calV,x}$.
Denote by
$$\pi : \{x\}^\bot \rightarrow \{x\}^\bot/\F x$$
the canonical projection. Just like in Section \ref{proofalternatingSection},
we obtain that $\pi(\calV x \cap \{x\}^\bot)$ is $\overline{b}_a$-orthogonal to $\pi(L_{\calV,x})$ and
that
$$\dim \bigl(\pi(L_{\calV,x})\bigr)=\dim L_{\calV,x}-1 \quad \text{and} \quad \dim \bigl(\pi(\calV x \cap \{x\}^\bot)\bigr)=\dim (\calV x \cap \{x\}^\bot).$$
Therefore,
$$\dim \calV \leq \dim \bigl(\pi(\calV x \cap \{x\}^\bot)\bigr)+\dim (\pi(L_{\calV,x}))+1+\dim (\calV \modu x).$$
Now, by Lemma \ref{atransformlemma}, we have
$$\dim \bigl(\pi(\calV x \cap \{x\}^\bot)\bigr)+\dim (\pi(L_{\calV,x})) \leq (n-2)+1,$$
whereas, by induction, we have
$$\dim (\calV \modu x) \leq (\nu-1)\bigl((n-2)-(\nu-1)\bigr)=(\nu-1)(n-\nu-1).$$
Hence,
$$\dim \calV \leq (\nu-1)(n-\nu-1)+n=\nu(n-\nu)+1.$$
As we have assumed that $\dim \calV>\nu(n-\nu)$, the above sequence of inequalities leads to
$$\dim \bigl(\pi(\calV x \cap \{x\}^\bot)\bigr)+\dim (\pi(L_{\calV,x}))=(n-2)+1.$$
Using Lemma \ref{atransformlemma} once more, we deduce that $n-2=2(\nu-1)+1$ and
that $\pi(L_{\calV,x})$ includes $(\Ker \overline{Q})^\bot$. The first result yields $n=2\nu+1$.

Finally, let $y \in (\Ker Q)^\bot$. In particular $y \in \{x\}^\bot$.
For all $z' \in \Ker \overline{Q}$, we have $z'=\pi(z)$ for some $z \in \{x\}^\bot \cap \Ker Q$, and hence
$\overline{b}(\pi(y),z')=\overline{b}(\pi(y),\pi(z))=b(y,z)=0$. This shows that $\pi(y) \in (\Ker \overline{Q})^\bot$,
and hence $\pi(y) \in \pi(L_{\calV,x})$. Using $\F x \subset L_{\calV,x}$, we conclude that
$y \in L_{\calV,x}$. Thus, $x \wedge_b y \in \calV$ for all $x \in \Ker Q \setminus \{0\}$ and all $y \in (\Ker Q)^\bot$.
\end{proof}

By Lemma \ref{perfectlemma}, we see that no generality is lost in assuming that $\F$ is perfect, an assumption we will
make in the remainder of the proof.

By Lemma \ref{perfectoddlemma}, we can take a basis $\bfB=(e_1,\dots,e_n)$ of $V$ with respect to which the matrix of $b$ equals
$$S=I_1 \oplus A \quad \text{where} \quad A:=\begin{bmatrix}
0_\nu & I_\nu \\
I_\nu & 0_\nu
\end{bmatrix}.$$
Note that $A$ is alternating and non-singular.

Denote by $\calM$ the vector space of all matrices that represent the elements of $\calV$ with respect to the basis $\bfB$:
it is a nilpotent subspace of $\calS_S$.
One checks that the $S$-symmetric matrices are those of the form
$$M=\begin{bmatrix}
\alpha(M) & (AC(M))^T \\
C(M) & K(M)
\end{bmatrix}$$
where $\alpha(M) \in \F$, $C(M) \in \F^{n-1}$ and $K(M)$ is $A$-symmetric.

\begin{claim}\label{claim2}
The space $\calM$ contains the matrix $\begin{bmatrix}
0 & (AX)^T \\
X & 0_{n-1}
\end{bmatrix}$ for all $X \in \F^{n-1}$.
\end{claim}

\begin{proof}
Note that $\Ker Q=\Vect(e_2,\dots,e_n)$ and $(\Ker Q)^\bot=\Vect(e_1)$.

Let $y \in \Vect(e_2,\dots,e_n)$. By Claim \ref{claim1} the operator $y \wedge_b e_1$ belongs to $\calV$.
Clearly, it maps $\{e_1\}^\bot$ into $\Vect(e_1)$, and it maps $e_1$ into $\Vect(e_2,\dots,e_n)$, and hence its matrix $N_y$ with respect to $\bfB$ has the form
$$N_y=\begin{bmatrix}
0 & (AC(N_y))^T \\
C(N_y) & 0_{n-1}
\end{bmatrix}.$$
It follows that the linear mapping $y \in \Ker Q \mapsto C(N_y) \in \F^{n-1}$ is injective, and as $\dim \Ker Q=n-1$ we deduce that it is
also surjective. The claimed statement follows.
\end{proof}

\begin{claim}\label{claim3}
For all $M \in \calM$, the matrix $K(M)$ is $A$-alternating and nilpotent, and $\alpha(M)=0$.
\end{claim}

\begin{proof}
Let $M \in \calM$. Set $X:=C(M)$. Claim \ref{claim2} shows that $M-\begin{bmatrix}
0 & (AX)^T \\
X & 0_{n-1}
\end{bmatrix}$ belongs to $\calM$, and hence
$$\begin{bmatrix}
\alpha(M) & [0]_{1 \times (n-1)} \\
[0]_{(n-1) \times 1} & K(M)
\end{bmatrix} \in \calM.$$
In particular, this matrix is nilpotent, which yields that $\alpha(M)=0$ and that $K(M)$ is nilpotent.
Using Claim \ref{claim2} once more, we find that, for every $X \in \F^{n-1}$, the matrix
$$\begin{bmatrix}
0 & (AX)^T \\
X & K(M)
\end{bmatrix}$$
belongs to $\calM$ and is therefore nilpotent. As we know that $K(M)$ is $A$-symmetric, Lemma \ref{extensionlemma} yields that $K(M)$ is $A$-alternating.
\end{proof}

We are now ready to conclude.
By the previous claim, every matrix $M$ of $\calM$ has the form
$$M=\begin{bmatrix}
0 & (A C(M))^T \\
C(M) & K(M)
\end{bmatrix}$$
and $K(\calM)$ is a nilpotent linear subspace of $\calA_A$. Here, the Witt index of $A$ equals $\nu$.
Thus point (b) of Theorem \ref{maintheo} shows that $\dim K(\calM) \leq \nu (\nu-1)$. As the linear mapping
$M \in \calM \mapsto (C(M),K(M)) \in \F^{n-1} \times K(\calM)$ is injective, we obtain
$$\dim \calV=\dim \calM \leq (n-1)+\dim K(\calV)=2\nu+\nu^2-\nu=\nu(n-\nu).$$
This actually contradicts one of our starting assumptions! Hence, we conclude that
$$\dim \calV \leq \nu(n-\nu).$$
This completes our inductive proof of point (a) of Theorem \ref{maintheo}.

\section*{Acknowledgements}

The author is profoundly grateful to Rachel Quinlan, who provided him with the example of Section \ref{specialexamplesection},
who proof-read an early version of the manuscript, and whose many helpful remarks contributed greatly to the quality of the present article.

\end{document}